\documentclass[10pt]{article}
\usepackage{graphicx}
\usepackage{amsmath}
\usepackage{amsfonts}
\usepackage{amssymb}
\usepackage{amsthm} 
\usepackage[english]{babel}
\usepackage{enumerate}
\usepackage[cp1250]{inputenc}
\usepackage[a4paper, left=3cm, right=3cm, top=3cm, bottom=3cm, headsep=1.2cm]{geometry}
\usepackage{todonotes}

\usepackage{subcaption}

\usepackage{caption}
\usepackage{color}
\definecolor{sepia}{cmyk}{0, 0.83, 1, 0.70}

\usepackage[plainpages=false,pdfpagelabels,bookmarksnumbered,%
 colorlinks=true,%
 linkcolor=sepia,%
 citecolor=sepia,%
 unicode]{hyperref}

\newtheorem{theorem}{Theorem}

\newtheorem{corollary}[theorem]{Corollary}

\theoremstyle{definition}

\newtheorem{definition}[theorem]{Definition}
\newtheorem{example}[theorem]{Example}

\newtheorem{remark}[theorem]{Remark}

\numberwithin{equation}{section} 
\numberwithin{theorem}{section}  

\makeatletter
\renewenvironment{proof}[1][\proofname]
{\par
	\pushQED{$\blacksquare$} 
	\normalfont\topsep6\p@\@plus6\p@\relax
	\trivlist
	\item[\hskip\labelsep\bfseries#1\@addpunct{.}]
	\ignorespaces}
{\popQED \endtrivlist\@endpefalse}
\makeatother

\DeclareMathOperator*{\interior}{int}
\DeclareMathOperator*{\ri}{ri}

\DeclareMathOperator*{\fix}{Fix}
\DeclareMathOperator*{\id}{Id}

\DeclareMathOperator*{\argmax}{argmax}

\begin{document}
\title{\textbf{Linear Convergence Rates for Extrapolated Fixed Point Algorithms}}
\author{Christian Bargetz\thanks{Department of Mathematics, University of Innsbruck, Technikerstra{\ss}e~13, 6020 Innsbruck, Austria, \protect\\ christian.bargetz@uibk.ac.at}
\and Victor I. Kolobov\thanks{Department of Computer Science, The Technion -- Israel Institute of Technology, 32000 Haifa, Israel, \protect\\ kolobov.victor@gmail.com}
\and Simeon Reich\thanks{Department of Mathematics, The Technion -- Israel Institute of Technology, 32000 Haifa, Israel, \protect\\sreich@tx.technion.ac.il}
\and Rafa\l\ Zalas\thanks{Department of Mathematics, The Technion -- Israel Institute of Technology, 32000 Haifa, Israel, \protect\\rzalas@tx.technion.ac.il}}
\maketitle

\begin{abstract}
We establish linear convergence rates for a certain class of extrapolated fixed point algorithms which are based on dynamic string-averaging methods in a real Hilbert space. This applies, in particular, to the extrapolated simultaneous and cyclic cutter methods. Our analysis covers the cases of both metric and subgradient projections.

\end{abstract}


\section{Introduction}
For a given family of nonempty, closed and convex subsets $C_i$ of a real Hilbert space  $\mathcal H$, $i\in I:=\{1,\ldots,M\}$, the \textit{convex feasibility problem} is to find a point $x\in C:=\bigcap_{i\in I} C_i$. In this paper we assume that $C\neq \emptyset$ and that $C_i=\fix U_i$ for a given cutter operator $U_i \colon \mathcal H \to \mathcal H$, $i\in I$. We recall that $U\colon\mathcal H\to\mathcal H$ is a \textit{cutter}, for example, a metric or a subgradinet projection, if $\fix U\neq\emptyset$ and $\langle x-Ux,z-Ux\rangle \leq 0$ for all $x\in\mathcal H$ and $z\in \fix U$. We consider the \textit{extrapolated string averaging} (ESA) method, which is a particular fixed point algorithmic framework of the form
\begin{equation}\label{intro:xk}
  x^0\in\mathcal{H} \qquad \text{and} \qquad
  x^{k+1} := x^k+\lambda_k\sigma_k(x^k)\left(T_k x^{k}-x^k\right), \quad k=0,1,2,\ldots,
\end{equation}
where $\lambda_k>0$ is a \textit{relaxation parameter}, $\sigma_k(\cdot)\colon\mathcal H \to (0,\infty)$ is an \textit{extrapolation functional} and $T_k\colon\mathcal H\to\mathcal H$ is a \textit{string averaging operator} which depends on a chosen subset of $U_1,\ldots,U_M$, that is,
\begin{equation}\label{intro:Tk}
  T_k:= \sum_{n=1}^{N_k}\omega_n^k Q_n^k,
\end{equation}
where $\omega_n^k \in (0,1]$, $\sum_n\omega_n^k=1$ and $Q_n^k:=\prod_{j\in J_n^k} U_j$ is a product of the operators $U_j$ along a nonempty ordered subset $J_n^k\subseteq I$ to which we refer as a \textit{string}.

The extrapolation $\sigma_k$ is intended to accelerate the convergence of method \eqref{intro:xk} in some instances of the problem. It is usually assumed to have values greater than or equal to one, as is done in this paper, although some authors allow smaller values of $\sigma_k$, which are then bounded away from zero. Without any loss of generality, one can assume that $\sigma_k(x):= 1$ whenever $T_kx=x$. If $\sigma_k(x):= 1$ for every $x\in\mathcal H$, then \eqref{intro:xk} becomes the basic, \textit{non-extrapolated} string averaging method.

The operator $T_k$ above is nothing but a convex combination of products of the operators $U_i$ along chosen strings $J_n^k$. The algorithmic structure of such an operator is presented in Figure \ref{fig:SA}. In the extreme cases, the string averaging operator can be reduced either to a cyclic cutter $T_k:=U_M\ldots U_1$ or to a simultaneous (parallel) cutter $T_k:=\sum_{i\in I_k} \omega_i^k U_i$, where $I_k\subseteq I$. According to \eqref{intro:xk}, we allow the structure of $T_k$ to change dynamically from iteration to iteration, which we explain in more detail in Section \ref{sec:methods}. We mention here only that we allow the block iterative framework, where each block $I_k:=J_1^k\cup\ldots\cup J_{N_k}^k$ may differ from $I$. This, when combined with the parallel structure of the operator $T_k$, provides a lot of flexibility in determining $T_k$ and thus in computing the next iteration.

\begin{figure}[tbp]
\centering
\includegraphics[bb=0 0 137.76 68.03, scale=1] {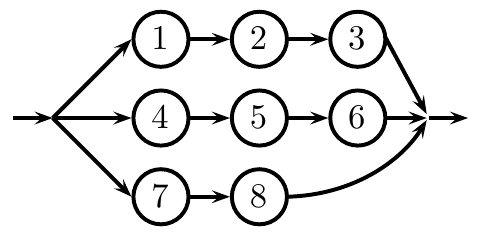}
\caption{Graphic representation of the string averaging operator $T=\frac 1 3 (U_3 U_2 U_1 + U_6 U_5 U_4 + U_8 U_7)$.}
\label{fig:SA}
\end{figure}

The assumption that each $U_i$ is a cutter provides an additional interpretation of the value $U_i x$. Indeed, $U_ix$ is nothing but the projection of $x$ onto the half-space $H_i(x):=\{z\mid \langle x-U_ix, z-U_ix \rangle \leq 0\} \supseteq C_i$, which for $x\notin C_i$ satisfies $x\notin H_i(x)$. Some authors have exploited this idea by projecting onto a certain closed and convex superset $C_i(x^k)\supseteq C_i$; see, for example, \cite{FlamZowe1990}, \cite{BauschkeBorwein1996}, \cite{Combettes1997} and \cite{ZhaoNgLiYao2017}. Although both approaches are theoretically different, the convergence analyses remain similar. The operator based approach, which appeared, for example, in  \cite{BauschkeNollPhan2015}, \cite{BargetzReichZalas2017} and \cite{CegielskiReichZalas2018}, enables us to extract abstract properties of the algorithmic operators. This is of independent interest and will be emphasized in the present paper.

In this paper we focus on the convergence properties and, in particular, on the linear convergence of the scheme \eqref{intro:xk}--\eqref{intro:Tk}. The convergence depends on the one hand on the definition of $\sigma_k$ and, on the other hand, follows from the regularity of the constraints $C_i$ and the operators $U_i$, $i\in I$; see Section \ref{sec:regularities} for the relevant definitions. We recall several known examples of $\sigma_k(\cdot)$ that guarantee weak and, in some cases, norm and even linear convergence. A more detailed overview of extrapolated simultaneous and cyclic cutter methods can be found, for example, in \cite{Cegielski2012}; see also \cite{BauschkeCombettesKruk2006}.

We begin with the simultaneous cutter methods for which
\begin{equation}\label{intro:sigma:SC}
  \sigma_k(x):=\frac{\sum_{i\in I_k}\omega_i^k \|U_i x-x\|^2}{\|\sum_{i\in I_k} \omega_i^k U_i x-x\|^2}\geq 1.
\end{equation}
Pierra considered \eqref{intro:sigma:SC} in \cite{Pierra1984} for the \textit{extrapolated parallel projection method}, where he established weak convergence of the sequence of iterates for $U_i=P_{C_i}$, $I_k=I$ and $\omega_i^k=1/M$. Moreover, under the bounded regularity of the family $\{C_i\mid i\in I\}$ the convergence was shown to be in norm. The \textit{extrapolated parallel subgradient projection} method ($U_i=P_{f_i}$) was introduced by Dos Santos in \cite{DosSantos1987} in $\mathbb R^d$. Since then one can find many extensions in the literature. For example, Combettes \cite{Combettes1996,Combettes1997} proposed the \textit{extrapolated method of parallel approximate projections} (EMPAP), where each $U_i$ was assumed to be the projection onto a closed and convex superset $C_i(x^k)\supseteq C_i$ and $\lambda_k \in [\varepsilon,2-\varepsilon]$ for some $\varepsilon\in (0,1)$. The method was shown to converge weakly under additional regularity of the approximate projections, that is, $\|U_i x^k-x^k\|\geq \delta_i d(x^k,C_i)$. Norm convergence, as in \cite{Pierra1984}, required bounded regularity of the family $\{C_i\mid i\in I\}$. Recently, Zhao et al. \cite{ZhaoNgLiYao2017} have proved that EMPAP converges linearly whenever the family of the sets is assumed to be boundedly linearly regular.

The extrapolated cyclic cutter ($M\geq 2$) appeared for the first time in \cite{CegielskiCensor2012} by Cegielski and Censor, where
\begin{equation}\label{intro:sigma:CC}
  \sigma_k(x):= \frac 1 2 +
    \frac{\sum_{i=1}^{M} \left\|U_i\ldots U_1 x-U_{i-1}\ldots U_1 x \right\|^2}
    {2\left\|U_M\ldots U_1 x-x\right\|^{2}}\geq \frac 1 2 + \frac 1 {2M}
\end{equation}
and $\lambda_k \in [\varepsilon,2-\varepsilon]$ for some $\varepsilon\in (0,1)$. The method was shown to converge weakly whenever each $U_i$ was weakly regular ($U_i-\id$ is demi-closed at 0), $i\in I$. This result was later extended by Nikazad and Mirzapour in \cite{NikazadMirzapour2015} to relaxed weakly regular cutters with a slightly modified $\sigma_k$. Extensive numerical tests for the extrapolated cyclic subgradient projection can be found in \cite{CegielskiNimana2018}.

The extrapolation formula for the general string averaging operator defined by \eqref{intro:Tk} can use both \eqref{intro:sigma:SC} and \eqref{intro:sigma:CC}. As far as we know, a natural extension based on \eqref{intro:sigma:SC} was for the first time proposed by Crombez \cite{Crombez2002}, where

\begin{equation}\label{intro:sigma:SA1}
  \sigma_k(x):=\frac{\sum_{n=1}^{N_k}\omega_n^k \|Q_n^k x-x\|^2}{\|\sum_{n=1}^{N_k}\omega_n^kQ_n^k x-x\|^2}\geq 1.
\end{equation}
Its convergence was investigated in the Euclidean space setting with continuous paracontractions $U_i$. Weak convergence in the Hilbert space setting involving $\alpha_i$-averaged operators $U_i$, $\alpha_i \in (0,1)$, was discussed in \cite{AleynerReich2008}, where a sufficient condition for norm convergence was also presented ($\interior C\neq \emptyset$ which implies bounded linear regularity). In particular, in the case of $(1/2)$-averaged operators (which are cutters), it was assumed that the relaxation parameter $\lambda_k\in [\varepsilon, 1+\frac 1 m-\varepsilon]$, where $m$ is the length of the longest string.

An extrapolation formula combining \eqref{intro:sigma:SC} and \eqref{intro:sigma:CC} was proposed by Nikazad and Mirzapour in \cite{NikazadMirzapour2017}. We present it here for cutters only, although it should be mentioned that it was formulated for $\alpha_i$-relaxed cutters, $\alpha_i\in (0,2)$:

\begin{equation}\label{intro:sigma:SA2}
  \sigma_k(x):=
  \frac
  {
    \sum\limits_{n=1}^{N_k} \omega_n^k
    \left(
      \left\|Q_n^kx-x \right\|^2
        + \sum\limits_{l=1}^{m_n^k}\|Q_{n,l}^kx-Q_{n,l-1}^{k}x\|^2
      \right)
  }
  {2 \|\sum_{n=1}^{N_k}\omega_n^kQ_n^k x-x\|^2}
  \geq \frac 1 2 + \frac 1 {2m},
\end{equation}
where $m_n^k$ is the length of the string $J_n^k$ and $Q_{n,l}^k$ is the product of the first $l$ operators along the string $J_n^k$ with $Q_{n,0}^k:=\id$. Method \eqref{intro:xk}--\eqref{intro:Tk}, when combined with \eqref{intro:sigma:SA2}, was shown to converge weakly while assuming that each $U_i$ is a weakly regular cutter, $\lambda_k\in [\varepsilon, 2-\varepsilon]$ and that there is no other dependence on $k$ within the iterations, that is, $T_k=T$, $\sigma_k=\sigma$ and $I_k=I$.

As far we know, the linear rate of convergence in the framework of the extrapolated string averaging cutter, including the extrapolated cyclic cutter, has not been investigated so far. The exception is the extrapolated simultaneous cutter \cite{ZhaoNgLiYao2017}, as we have already mentioned above. Nevertheless, there are a few known results which deal with the linear rate of the non-extrapolated version of method \eqref{intro:xk}--\eqref{intro:Tk}. For example, Bargetz et al. \cite{BargetzReichZalas2017} have shown that the dynamic string averaging projection method converges linearly whenever the family of sets is boundedly linearly regular. A linear rate of convergence for boundedly linearly regular cutters has recently been established in \cite{CegielskiReichZalas2018}, although the result required the additional assumption that each subfamily of $\{C_i\mid i\in I\}$ is boundedly linearly regular; see \cite[Example 5.7, Theorem 6.2 ]{CegielskiReichZalas2018}. The above-mentioned assumption guarantees that each one of the operators $Q_n^k$ that appears in \eqref{intro:Tk} and, in consequence $T_k$, is boundedly linearly regular; see Remark \ref{rem:LR}. As was shown in \cite{ReichZaslavski2014}, this may not be the case for some families of sets.

There are many other works which deal with the non-extrapolated string averaging method \eqref{intro:xk}--\eqref{intro:Tk} although linear convergence was not discussed therein. In particular, a static version of the string averaging method  \eqref{intro:xk}--\eqref{intro:Tk}  was introduced in \cite{CensorElfvingHerman2001} for the metric and Bregman projections in Euclidean space. A dynamic variant appeared, for example, in \cite{CensorZaslavski2013} and in \cite{ReichZalas2016}. Other works related to string averaging methods are, for example, \cite{CensorTom2003}, \cite{BauschkeMatouskovaReich2004}, \cite{CensorSegal2007}, \cite{ButnariuDavidiHermanKazantsev2007}, \cite{ButnariuReichZaslavski2008}, \cite{Zaslavski2014}, \cite{CensorMansour2016}, \cite{Zaslavski2016} and \cite{NikazadAbbasiMirzapour2016}.

We emphasize here that a linear rate of convergence is known in the framework of the non-extrapolated simultaneous and cyclic methods with boundedly linearly regular operators and families of sets; see, for example, \cite{BauschkeNollPhan2015} or \cite{BorweinLiTam2017} with averaged operators, or \cite{BauschkeBorwein1996} and \cite{KolobovReichZalas2017} with cutters. Here also no additional linear regularity of subfamilies is required.

\subsection*{Contribution and organization of the paper}
The main contribution of this paper is the formulation of sufficient conditions for linear convergence of the extrapolated string averaging method \eqref{intro:xk}--\eqref{intro:Tk} in terms of bounded linear regularity of the algorithmic operators $U_i$ and the family of sets $\{C_i\mid i\in I\}$. Following \cite[Theorem 9]{BargetzReichZalas2017}, which is formulated for projections only, we show that the linear rate of convergence holds without any additional regularities of the subfamilies of $\{C_i\mid i\in I\}$, which is required in \cite{CegielskiReichZalas2018}; see the explanation above. Our slightly modified extrapolation functional $\sigma_k\colon \mathcal H \to [1,\infty)$, which is based on \eqref{intro:sigma:SA2}, when combined with the relaxation parameters $\lambda_k\in[\underline \lambda, \overline \lambda]\subseteq (0, 1+ 1/m)$ allows us to unify the examples of $\sigma_k$ presented in \eqref{intro:sigma:SC}--\eqref{intro:sigma:SA2}; see Remark \ref{rem:lambda}. Our estimate for the linear convergence rate improves the one presented in \cite{BargetzReichZalas2017, CegielskiReichZalas2018} and, when reduced to projections only, coincides with known results for cyclic and simultaneous projections. For completeness, we also discuss weak and norm convergence. We emphasize here that in all the three types of convergence we allow a dynamically changing sequence of operators $\{T_k\}$ which is not the case in \cite{CegielskiCensor2012, NikazadMirzapour2017, NikazadMirzapour2015}.

In addition to the linear rate of convergence, we also establish new results related to the linear regularity of operators, which we believe are of independent interest.

Finally, we present results of simple numerical simulations, which show that extrapolation may indeed be considered an acceleration technique for some instances of the string averaging method.

Our paper is organized as follows. In Section \ref{sec:preliminaries} we introduce our notations, definitions and useful facts. In Section \ref{sec:ExtrOper} we discuss basic properties of extrapolated operators (Theorem \ref{th:ESA}), whereas in Section \ref{sec:methods} we use these properties to derive the main result of this paper (Theorem \ref{th:main}). In the last section we present the results of our numerical simulations.

\section{Preliminaries}\label{sec:preliminaries}
In this paper $\mathcal{H}$ always denotes a real Hilbert space. For a sequence $\{x^k\}_{k=0}^\infty$ in $\mathcal{H}$ and a point $x^\infty\in\mathcal{H}$, we use the notations $x^k\rightharpoonup x^\infty$ and $x^k\to x^\infty$ to indicate that $\{x^k\}_{k=0}^{\infty}$ converges to $x^\infty$ weakly and in norm, respectively.

Given a nonempty, closed and convex set $C\subseteq\mathcal{H}$, we denote by $P_C\colon\mathcal{H}\to\mathcal{H}$ the \emph{metric projection onto $C$}, that is, the operator which maps $x\in\mathcal{H}$ to the unique point in $C$ closest to $x$. The operator $P_C$ is well defined for such sets $C$ and it is not difficult to see that it is nonexpansive; see, for example, \cite[Proposition 4.8]{BauschkeCombettes2011}, \cite[Theorem 2.2.21]{Cegielski2012} or \cite[Theorem 3.6]{GoebelReich1984}. We denote by $ d(x,C) := \inf \{\|x-z\|\mid z\in C\}$ the distance of $x$ to $C$.

For a given operator $U\colon\mathcal H\to\mathcal H$, we denote by $\fix U$ the set of fixed points of $U$. We denote by $U_\alpha$ the ($\alpha$-)\textit{relaxation of} $U$ defined by $U_\alpha x:=x+\alpha(Ux-x)$ for each $x\in \mathcal H$, where $\alpha >0$. We use the same symbol $U_\alpha$ for the \textit{generalized relaxation}, where $\alpha\colon\mathcal H \to (0,\infty)$. In this case $U_\alpha x:=x+\alpha(x)(Ux-x)$ for each $x\in\mathcal H$. If $\alpha(x)\geq 1$ for each $x\in \mathcal H$, then we say that $U_\alpha$ is an \textit{($\alpha$-)extrapolation} of $U$.

\subsection{Quasi-nonexpansive operators}
\begin{definition}\label{def:QNE}
Let $U\colon\mathcal H\rightarrow\mathcal H$ be an
operator with $\fix U\neq \emptyset$. We say that $U$ is
\begin{enumerate}[(i)]
\item \textit{quasi-nonexpansive} (QNE) if for all $x\in\mathcal H$ and all $z\in\fix U$,
\begin{equation}
\| U x -z\|\leq\| x-z\|;\label{eq:qne}%
\end{equation}

\item $\rho$\textit{-strongly quasi-nonexpansive} ($\rho$-SQNE), where $\rho\geq 0$, if for all $x\in\mathcal H$ and all $z\in\fix U$,
\begin{equation}
\| Ux-z\|^2\leq\| x-z\|^2-\rho\| U
x -x\|^2;\label{eq:sqne}
\end{equation}

\item a \textit{cutter} if for all $x\in\mathcal H$ and all $z\in\fix U$,
\begin{equation}
\langle z-U x ,x-U x \rangle\leq 0\label{eq:cutter}.
\end{equation}
\end{enumerate}
\end{definition}

Below we recall some properties of quasi-nonexpansive operators that will be used in the sequel. A more comprehensive overview can be found in \cite[Chapter 2]{Cegielski2012}.

\begin{theorem}\label{th:SQNE}
  Let $U\colon\mathcal H \rightarrow \mathcal H$ be such that $\fix U\neq\emptyset$ and let $\rho\geq 0$. The following conditions are equivalent:
  \begin{enumerate}[(i)]
    \item $U$ is $\rho$-SQNE.
    \item $\id+\frac{1+\rho}{2}(U-\id)$ is a cutter.
    \item for each $x\in\mathcal H$ and $z\in \fix U$, we have
    \begin{equation}\label{eq:th:SQNE}
      \langle z-x, Ux-x\rangle \geq \frac{1+\rho}{2}\|Ux-x\|^2.
    \end{equation}
  \end{enumerate}
\end{theorem}

\begin{proof}
  See \cite[Corollary 2.1.43 and Section 2.1.30]{Cegielski2012}.
\end{proof}

\begin{corollary}\label{th:SQNElambda}
Let $U\colon\mathcal H \to \mathcal H$ be $\rho$-SQNE, where $\rho>0$, and let $\alpha\colon\mathcal H \to (0,\infty)$. Then, for each $x\in\mathcal H$ and $z\in \fix U$, the generalized relaxation $U_\alpha$ satisfies
\begin{equation}\label{eq:th:SQNElambda}
  \|U_\alpha x-z\|^2\leq \|x-z\|^2-\left(\frac \rho {\alpha(x)}+\frac{1-\alpha(x)}{\alpha(x)}\right)\|U_\alpha x-x\|^2.
\end{equation}
\end{corollary}

\begin{proof}
Since $U$ is $\rho$-SQNE, by Theorem \ref{th:SQNE} (iii), we have
\begin{equation}\label{}
  \langle x-z, Ux-x\rangle \leq -\frac{\rho+1}{2}\|Ux-x\|^2.
\end{equation}
Moreover,
\begin{equation}\label{}
  U_\alpha x-x = \alpha(x)(Ux-x).
\end{equation}
Consequently,
\begin{align}\nonumber
  \|U_\alpha x-z\|^2
  & = \|x-z+ \alpha(x)(Ux-x)\|^2\\ \nonumber
  & = \|x-z\|^2 + 2\alpha (x)\langle x-z, U x-x\rangle + \alpha^2(x)\|Ux-x\|^2\\ \nonumber
  & \leq \|x-z\|^2 -2\alpha (x)\frac{\rho+1}{2}\|Ux-x\|^2 + \alpha^2(x)\|Ux-x\|^2\\
  & = \|x-z\|^2 - \alpha^2(x)\left( \frac \rho {\alpha(x)}
  + \frac{1-\alpha(x)}{\alpha(x)}\right)\|Ux-x\|^2,
\end{align}
which completes the proof.
\end{proof}

\begin{remark}
Observe that inequality \eqref{eq:th:SQNElambda} becomes significant when
\begin{equation}\label{}
  \frac \rho {\alpha(x)}+\frac{1-\alpha(x)}{\alpha(x)} > 0
\end{equation}
which holds for any $\alpha(x)\in(0,1+\rho)$. We will apply \eqref{eq:th:SQNElambda} in this form in the sequel; see, for example, Theorems \ref{th:ESA} or \ref{th:main}.
\end{remark}

\begin{theorem} \label{th:SQNE2}
Let $U_i\colon\mathcal{H}\rightarrow \mathcal{H}$ be $\rho_{i}$-SQNE, $i=1,2,...,M$, where $M\geq 1$, $\rho_i> 0$ and assume that $C=\bigcap_{i=1}^M \fix U_i\neq \emptyset$.

\begin{enumerate}[(i)]
  \item Let $U:=\sum_{i=1}^M\omega_i U_i$, where $\omega _{i}> 0$ and $\sum_{i=1}^M\omega_i=1$. Then $U$ is $\rho $-SQNE, $\rho =\min_{i}\rho _{i}$ and $\fix U=C$. Moreover, if each $U_i$ is a cutter ($\equiv$ 1-SQNE), then for any $x\in \mathcal{H}$ and $z\in C$, we have
      \begin{equation}\label{th:SQNE2:ineq1}
        \langle Ux-x, z-x\rangle \geq \sum_{i=1}^{M} \omega_i\|U_ix-x\|^2.
      \end{equation}

  \item Let $U:=U_M\ldots U_1$. Then $U$ is $\rho$-SQNE,     $\rho=\min_{i}\rho_{i}/M$ and $\fix U=C$. Moreover, if each $U_i$ is a cutter ($\equiv$ 1-SQNE), then for any $x\in \mathcal{H}$ and $z\in C$, we have
      \begin{equation}\label{th:SQNE2:ineq2}
        \langle Ux-x, z-x\rangle \geq \frac 1 2 \|Ux-x\|^2 + \frac 1 2 \sum_{i=1}^M \|U_i\ldots U_1 x-U_{i-1}\ldots U_1 x\|^2,
      \end{equation}
      where we set $U_{i-1}\ldots U_1 x:=x$ for $i=1$.
\end{enumerate}

\end{theorem}

\begin{proof}
For (i), see \cite[Theorem 2.1.50]{Cegielski2012}. Inequality \eqref{th:SQNE2:ineq1} follows directly from Theorem \ref{th:SQNE} (iii) and definition of $U$.
For (ii), see \cite[Theorem 2.1.48]{Cegielski2012}. Inequality \eqref{th:SQNE2:ineq2} follows from \cite[Lemma 4.10.1]{Cegielski2012} and the last equality in its proof.
\end{proof}

\subsection{Regular sets and regular operators}\label{sec:regularities}
\begin{definition}\label{def:BR:set}
  Given a set $S\subseteq\mathcal{H}$, a family $\{C_1,\ldots,C_M\}$
   of convex and closed sets $C_i\subseteq\mathcal{H}$ with a nonempty intersection $C:=\bigcap_{i=1}^M C_i$ is called
  \begin{enumerate}[(i)]
  \item \emph{regular over $S$} if for every sequence $\{x^k\}\subseteq S$, we have
    \begin{equation}
      \max_{i=1,\ldots, M} d(x^k,C_i) \to_k 0\quad \Longrightarrow\quad d(x^k,C) \to_k 0;
    \end{equation}
  \item \emph{$\kappa_S$-linearly regular over $S$} if the inequality
    \begin{equation}
      d(x,C) \leq \kappa_S \max_{i=1,\ldots,M} d(x,C_i)
    \end{equation}
    holds for all $x\in S$ and some constant $\kappa_S>0$.
  \end{enumerate}
  We say that the family $\{C_1,\ldots,C_M\}$ is \emph{boundedly (linearly) regular} if it is ($\kappa_S$-linearly) regular over every bounded subset $S\subseteq\mathcal{H}$. We say that $\{C_1,\ldots,C_M\}$ is \emph{(linearly) regular} if it is ($\kappa_S$-linearly) regular over $S=\mathcal H$.
\end{definition}

\begin{example}\label{example:regularSets}
Let $C_i\subseteq\mathcal H$ be as in Definition \ref{def:BR:set} and let $\mathcal C:=\{C_1,\ldots,C_M\}$. Following \cite[Fact 5.8]{BauschkeNollPhan2015}, we recall some examples of regular families of sets:
\begin{enumerate}[(i)]
\item If $\dim\mathcal H<\infty$, then $\mathcal C$ is boundedly regular.
\item If $C_M\cap\interior\bigcap_{i=1}^{M-1}C_i\neq\emptyset$, then $\mathcal C$ is boundedly linearly regular.
\item If each $C_i$ is a half-space, then $\mathcal C$ is linearly regular.
\item If $\dim\mathcal H<\infty$, $C_i$ is a half-space for $i=1,\ldots,L$, and $\bigcap_{i=1}^LC_i \cap \bigcap_{i=L+1}^M\ri C_i\neq\emptyset$, then $\mathcal C$ is boundedly linearly regular.
\item If each $C_i$ is a closed subspace, then $\mathcal C$ is linearly regular if and only if $\sum_{i=1}^MC_i^\perp$ is closed.
\end{enumerate}
\end{example}

For more information regarding regular families of sets see \cite{ReichZaslavski2014b}.

\begin{definition}\label{def:BR:oper}
Let $U:\mathcal H\rightarrow\mathcal H$ be an
operator with a fixed point, that is, $\fix U \neq \emptyset$, and let $S\subseteq\mathcal H$ be nonempty. We say that the operator $U$ is
\begin{enumerate}[(i)]
\item \textit{weakly regular} over $S$ if for any sequence $\{x^{k}\}_{k=0}^\infty \subseteq S$ and $x^\infty\in \mathcal H$,
\begin{equation}
\label{eq:def:DC}
\left .
\begin{array}{l}
x^{k}\rightharpoonup x^\infty\\
U x^k-x^k\rightarrow 0
\end{array}
\right\}\quad\Longrightarrow\quad x^\infty\in \fix U;
\end{equation}

\item \textit{regular} over $S$ if for any sequence $\{x^{k}\}_{k=0}^\infty \subseteq S$,
\begin{equation} \label{eq:def:BR:oper}
\lim_{k\rightarrow\infty}\| Ux^k-x^k\| =0\quad\Longrightarrow\quad \lim_{k\rightarrow\infty}d(x^k,\fix U)=0;
\end{equation}

\item \textit{linearly regular} over $S$ if there is $\delta_S>0$ such that for every $x\in S$,
\begin{equation} \label{eq:def:BLR:oper}
\delta_S d(x,\fix U)\leq \|Ux-x\|.
\end{equation}
\end{enumerate}
If any of the above regularity conditions holds for every subset $S\subseteq\mathcal H$, then we simply omit the phrase ``over $S$". If the same condition holds when restricted to bounded subsets $S\subseteq\mathcal H$, then we precede the term with the adverb \textit{boundedly}. Since there is no need to distinguish between boundedly weakly and weakly regular operators, we call both weakly regular.
\end{definition}

Clearly, weakly regular operators are those for which $U-\id$ is demi-closed at zero and they go back to \cite{BrowderPetryshyn1966} and \cite{Opial1967}. Regular operators appeared already in \cite[Theorem 1.2]{PetryshynWilliamson1973} whereas linearly regular operators can be found in \cite[Theorem 2]{Outlaw1969}. For a detailed historical overview of regular operators, we refer the reader to \cite{KolobovReichZalas2017} and \cite{CegielskiReichZalas2018}. We mention here only a few works where these operators appeared implicitly or explicitly; see, for example, \cite{BauschkeBorwein1996}, \cite{KiwielLopuch1997}, \cite{AoyamaKohsaka2014}, \cite{CegielskiZalas2013}, \cite{CegielskiZalas2014}, \cite{BauschkeNollPhan2015}, \cite{Cegielski2015}, \cite{ReichZalas2016}, \cite{BorweinLiTam2017} and \cite{ReichSalinas2017}.

Regular families of sets and regular operators are related in the sense that $U:=P_{C_{i(x)}}x$, where $i(x):=\argmax_{i=1,\ldots,M} d(x,C_i)$, is (linearly) regular over $S$ if and only if the family $\{C_1,\ldots,C_M\}$ is (linearly) regular over $S$. This was observed, for example, in \cite[Remark 2.13]{KolobovReichZalas2017}.

 It is not difficult to see that (bounded) linear regularity implies (bounded) regularity. One can also show that (bounded) regularity implies weak regularity. On the other hand, it turns out that weak regularity implies bounded regularity in $\mathcal H=\mathbb R^n$; see, for example \cite[Proposition 4.1]{CegielskiZalas2014} or \cite[Theorem 4.3]{CegielskiReichZalas2018}. This leads to the following two examples:

\begin{example}[Nonexpansive mapping]
  Let $U\colon\mathcal H\to\mathcal H$ be nonexpansive with $\fix U\neq\emptyset$. Then $U$ is weakly regular due to the demi-closednes principle \cite{Opial1967}. If $\mathcal H=\mathbb R^n$, then $U$ is boundedly regular \cite[Proposition 4.1, Corollary 4.2]{CegielskiZalas2014}. Moreover, in this case there is a more explicit connection between $d(x,\fix U)$ and $\|Tx-x\|$. To be specific, by \cite[Theorem 3]{LukeThaoTam2018}, for every bounded subset $S\subseteq \mathbb R^n$ which intersects $\fix U$, there is a bounded, increasing function $f_S\colon \mathbb R_+ \to \mathbb R_+$, right-continuous at $t=0$ with $f(0)=0$, which for all $x\in S$ satisfies
  \begin{equation}\label{}
    d(x,\fix U)\leq f_S(\|Tx-x\|).
  \end{equation}
\end{example}

\begin{example}[Subgradient projection]
Following \cite[Example 2.15]{KolobovReichZalas2017} or \cite[Example 3.5]{CegielskiReichZalas2018}, let $f\colon\mathcal H \to \mathbb R$ be a convex continuous function. Let $\partial f(x):=\{g\in\mathcal H \mid \forall y\in \mathcal H \ \langle g,y-x\rangle \leq f(y)-f(x) \}$ be the \textit{subdifferential} of $f$ at $x$ which, by the assumptions on $f$, is nonempty. Suppose that $S(f,0):=\{x\mid f(x)\leq 0\}\neq \emptyset$. For each $x$, we fix a subgradient $g(x)\in \partial f(x)$ and define the \textit{subgradient projection} by
\begin{equation}\label{}
  P_fx:=
  \begin{cases}
    x-\frac{f(x)}{\|g(x)\|^2}g(x) & \mbox{if } f(x)>0 \\
    x, & \mbox{otherwise}.
  \end{cases}
\end{equation}
If $f$ is Lipschitz continuous on bounded sets, then $P_f$ is weakly regular. In particular, if $\mathcal H=\mathbb R^d$, then $P_f$ is boundedly regular. In addition, if $f(x)<0$ for some $x$, then $P_f$ is boundedly linearly regular. It may also happen that the subgradient projection is not boundely regular; see \cite[Example 3.6]{CegielskiReichZalas2018}.
\end{example}

\begin{remark}[Relation between $\rho$ and $\delta_S$]\label{rem:delta}
  Observe that if $U$ is $\rho$-SQNE and $\delta_S$-linearly regular over some nonempty subset $S\subseteq\mathcal H$, then $\rho \delta_S^2\leq 1$. In particular, if $U$ is a cutter, then $\delta_S\leq 1$. Indeed, it suffices to substitute $z=P_{\fix U}x$ in the inequality below, which holds true for each $x\in S$ and $z\in \fix U$:
  \begin{equation}\label{}
    \rho \delta_S^2 d^2(x,\fix U) \leq \rho \|Ux-x\|^2\leq \|x-z\|^2.
  \end{equation}
  On the other hand, the first inequality holds true with any $\delta_S>0$ whenever $x\in \fix U$. Without any loss of generality, we can assume in this case that again $\rho \delta^2_S\leq 1$.
\end{remark}

\subsection{Fej\'{e}r monotone sequences}
\begin{definition}
Let $C\subseteq\mathcal H$ be a nonempty, closed and convex set, and let $\{x^k\}_{k=0}^\infty$ be a sequence in $\mathcal H$. We say that $\{x^k\}_{k=0}^\infty$ is \textit{Fej\'er monotone} with respect to $C$ if
\begin{equation}
\|x^{k+1}-z\|\leq\|x^k-z\|
\end{equation}
for all $z\in C$ and every integer $k=0,1,2,\ldots$.
\end{definition}

\begin{theorem} \label{th:Fejer}
Let the sequence $\{x^k\}_{k=0}^\infty$ be Fej\'er monotone with respect to $C$. Then
\begin{enumerate}[(i)]
\item $\{x^k\}_{k=0}^\infty$ converges weakly to some point $x^\infty\in C$ if and only if all its weak cluster points lie in $C$;

\item $\{x^k\}_{k=0}^\infty$ converges strongly to some point $x^\infty\in C$ if and only if
    \begin{equation}
    \lim_{k\rightarrow\infty}d(x^k,C) =0;
    \end{equation}

\item if there is some constant $q\in(0,1)$ such that $d(x^{k+1},C)\leq q d(x^k,C)$ holds for every $k=0,1,2,\ldots$, then $\{x^k\}_{k=0}^\infty$ converges linearly to some point $x^\infty\in C$ and
\begin{equation}
\|x^k-x^\infty\|\leq 2d(x^0,C)q^k;
\end{equation}

\item if $\{x^{ks}\}_{k=0}^\infty$ converges linearly to some point $x^\infty\in C$, that is, $\|x^{ks}-x^\infty\|\leq c q^k$ for some constants $c>0$, $q\in(0,1)$ and integer $s$, then the entire sequence $\{x^k\}_{k=0}^\infty$ converges linearly and moreover,
\begin{equation}
\|x^k-x^\infty\|\leq \frac{c}{(\sqrt[\scriptstyle{s}]{q})^{s-1}} \left(\sqrt[\scriptstyle{s}]{q}\right)^k;
\end{equation}

\item if $\{x^{k}\}_{k=0}^\infty$ converges strongly to some point $x^\infty\in C$, then $\|x^k-x^\infty\|\leq 2d(x^k,C)$ for every $k=0,1,2,\ldots$.
\end{enumerate}
\end{theorem}
\begin{proof}
See, for example, \cite[Theorem 2.16 and Proposition 1.6]{BauschkeBorwein1996}.
\end{proof}

\section{Extrapolated Operators}\label{sec:ExtrOper}

\begin{theorem}[Extrapolated String Averaging]\label{th:ESA}
  Let $U_i\colon\mathcal{H}\to\mathcal{H}$ be a cutter, $i\in I:=\{1,\ldots,M\}$, such that $C:=\bigcap_{i\in I} \fix U_i \neq \emptyset$. Let $\{J_1,\ldots,J_N\}$ be a family of strings satisfying $I=J_1\cup\ldots\cup J_N$. Let $T$ be the string averaging cutter operator defined by
  \begin{equation}\label{th:ESA:U}
    T:=\sum_{n=1}^{N} \omega_n Q_n,
  \end{equation}
  where $\omega_n\in(0,1]$ with $\sum_{n=1}^{N}\omega_n=1$ and $Q_n:=\prod_{j\in J_n} U_j$. Moreover, let $T_{\lambda\sigma}\colon\mathcal H \to \mathcal H$ be defined by
  \begin{equation}\label{eq:ESA:Usigma}
  T_{\lambda\sigma} x:= x + \lambda\sigma(x)(Tx-x),
  \end{equation}
  where the relaxation parameter $\lambda\in (0,1+ \frac 1 m)$, the extrapolation functional $\sigma\colon\mathcal H \to [1,\infty)$ satisfies $1\leq \sigma(x) \leq \tau(x)$ for all $x\in\mathcal H$ and
  \begin{equation}\label{eq:ESA:tau}
    \tau(x):=
    \begin{cases}  \displaystyle
      \frac
      {
      \frac{m}{m+1}\sum\limits_{n=1}^{N} \omega_n
      \left(
      \left\|Q_nx-x \right\|^2
      +\sum\limits_{l=1}^{m_n}\|Q_{n,l}x-Q_{n,l-1}x\|^2
      \right)
      }
      {
      \left\|T x-x\right\|^{2}
      } & \text{if}\ x\notin C\\
      \ 1 & \text{otherwise},
    \end{cases}
  \end{equation}
  where $m_n:=|J_n|$, $m:=\max_{1\leq n\leq N} m_n$ and for each string $J_n:=(i_{1},\ldots,i_{m_n})$, we use the notations $Q_{n,l}:=\prod_{j=1}^{l} U_{i_j}$ and $Q_{n,0}:=\id$. Then the following statements hold:
  \begin{enumerate}[(i)]
  \item $T_{\lambda\sigma}$ is $\rho$-SQNE, where $\fix T_{\lambda\sigma}=C$ and
  \begin{equation}\label{eq:ESA:rho}
    \rho:=\frac 1 {\lambda m} + \frac{1-\lambda}{\lambda}>0.
  \end{equation}
  Moreover, for all $x\in \mathcal H$ and $z\in C$, we have
  \begin{equation}\label{eq:ESA:ineq1}
    \Theta(x)
    \sum_{n=1}^{N}\omega_n \sum_{l=1}^{m_n}\big\|Q_{n,l}x-Q_{n,l-1}x\big\|^2
    \leq  \|x-z\|^2 - \|T_{\lambda\sigma} x- z\|^2,
  \end{equation}
  where
  \begin{equation}\label{eq:ESA:Theta}
    \Theta(x):=
    \lambda\sigma (x)
    \frac{1+m-m\lambda\frac{\sigma(x)}{\tau(x)}}
         {1+m-\frac{m}{\tau(x)}}
    \geq \lambda \left(1-\lambda \frac {m} {1+m}\right)>0.
  \end{equation}
  \item If for every $i\in I$, the operator $U_i$ is $\delta_i$-linearly regular, $\delta_i\in(0,1]$, on the ball $B(z,r)$ for some $z\in C$ and some $r>0$, then for all $x\in B(z,r)$, we have
      \begin{equation}\label{eq:ESA:ineq2}
        \delta^2 \max_{i\in I} d^2(x,\fix U_i) \leq \sum_{n=1}^{N} \sum_{l=1}^{m_n} m_n \big\|Q_{n,l}x-Q_{n,l-1}x\big\|^2,
      \end{equation}
      where $\delta:=\min_{i\in I} \delta_i$.

  \item If, in addition, the family $\{\fix U_i\mid i\in I\}$ is $\kappa$-linearly regular over $B(z,r)$, then $T_{\lambda\sigma}$ also satisfies
    \begin{equation}\label{eq:ESA:ineq3}
      \Theta(x) \frac{\omega\delta^2}{2m\kappa^2}  d\big(x,C\big)
      \leq \|T_{\lambda\sigma} x-x\|
    \end{equation}
    for all $x\in B(z,r)$, where $\omega:=\min_{n=1,\ldots, N}\omega_n$. Moreover, in this case
    \begin{equation}\label{eq:ESA:ineq4}
		\Theta(x)
    \leq \frac{m\kappa^2}{\omega\delta^2 }.
		\end{equation}
  \end{enumerate}
\end{theorem}

\begin{remark}\label{rem:tau}
  Observe that we can always choose $\sigma(x)$ such that $1\leq \sigma (x)\leq\tau(x)$ since $\tau(x)\geq 1$. Indeed, by the convexity of $\|\cdot\|^2$, we have
  \begin{equation} \label{eq:rem:tau:1}
    \|Tx-x\|^2 \leq \sum_{n=1}^{N}\omega_n\|Q_nx-x\|^2.
  \end{equation}
  Moreover, if $m_n\geq 2$, then
  \begin{equation} \label{eq:rem:tau:2}
    \|Q_nx-x\|^2
    \leq \left( 1\cdot \sum_{l=1}^{m_n}\|Q_{n,l}x-Q_{n,l-1}x\|\right)^2
    \leq m_n\sum_{l=1}^{m_n} \|Q_{n,l}x-Q_{n,l-1}x\|^2,
  \end{equation}
   which follows from the Cauchy-Schwarz inequality in $\mathbb{R}^{m_n}$. On the other hand, \eqref{eq:rem:tau:2} becomes an equality if $m_n=1$ and then it is equal to $\|U_ix-x\|^2$ for some $i\in I$. Consequently, $\tau(x)$ has to be greater than or equal to one.
\end{remark}

\begin{proof}[Proof of Theorem \ref{th:ESA}]
It is not difficult to see that, by Theorem \ref{th:SQNE2}, $C=\fix T=\fix T_{\lambda\sigma}$. From now on we assume that $x\in\mathcal H\setminus C$ and $z\in C$. Define $\alpha(x):=\lambda\sigma(x)/\tau(x)$, and observe that $0<\alpha(x)\leq 1+\frac 1 m$ and
\begin{equation}\label{}
  T_{\lambda\sigma} x=x+\lambda\frac{\sigma(x)}{\tau(x)}(T_\tau x-x)=(T_\tau)_{\alpha} x.
\end{equation}

\textit{Part (i).} In order to show that $T_{\lambda\sigma}$ is $\rho$-SQNE, it suffices in view of the above lines and Corollary \ref{th:SQNElambda}, to show that $T_\tau$ is $1/m$-SQNE. By Theorem \ref{th:SQNE2}, we have
\begin{equation}\label{pr:th:ESA:s1:productIneq1}
  \langle z-x, Q_n x-x \rangle
  \geq \frac 1 2 \|Q_nx-x\|^2 + \frac 1 2 \sum_{l=1}^{m_n} \|Q_{n,l}x-Q_{n,l-1}x\|^2.
\end{equation}
Note that \eqref{pr:th:ESA:s1:productIneq1} also holds true for the case $m_n=1$; compare with Theorem \ref{th:SQNE}(iii). Consequently, we have
\begin{align}\label{pr:th:ESA:s1:1} \nonumber
  \langle z-x, T_\tau x-x \rangle
  & = \tau(x) \langle z-x, T x-x \rangle
  =\tau(x)\sum_{n=1}^N\omega_n \langle z-x, Q_n x-x \rangle\\ \nonumber
  & \geq  \frac{\tau(x)}{2}\sum_{n=1}^N\omega_n \left( \|Q_nx-x\|^2 + \sum_{l=1}^{m_n} \|Q_{n,l}x-Q_{n,l-1}x\|^2\right)\\
  & = \frac{m+1}{2m}\tau^2(x)\|Tx-x\|^2
  = \frac{1+\frac 1 m}{2}\|T_\tau x-x\|^2,
\end{align}
which, again by Theorem \ref{th:SQNE} (iii), shows that $T_\tau$ is $(1/m)$-SQNE.

Now we show that inequality \eqref{eq:ESA:ineq1} holds true. By Corollary \ref{th:SQNElambda} applied to $T_{\lambda\sigma}=(T_\tau)_\alpha$, we obtain
\begin{align}\label{pr:th:ESA:s1:2}\nonumber
  \|T_{\lambda\sigma} x-z\|^2 &\leq \|x-z\|^2
  -\left(\frac 1 {m\alpha(x)} + \frac{1-\alpha(x)}{\alpha(x)}\right)
  \|T_{\lambda\sigma} x- x\|^2\\ \nonumber
  & = \|x-z\|^2
  -\lambda\sigma(x)\frac{1+m-m\alpha(x)}{m} \tau (x)\|T x- x\|^2 \\ \nonumber
  & = \|x-z\|^2
  -\lambda\sigma(x)\frac{1+m-m\alpha(x)}{m+1} \left(\sum_{n=1}^{N} \omega_n \left\|Q_nx-x \right\|^2 \right.\\
  & \quad+\left. \sum_{n=1}^{N} \omega_n \sum_{l=1}^{m_n}\|Q_{n,l}x-Q_{n,l-1}x\|^2\right),
\end{align}
which also holds true if $m_n=1$ for some $n$. On the other hand, by the convexity of $\|\cdot\|^2$ and the definition of $\alpha$ and $\tau$, we have
\begin{align}\label{pr:th:ESA:s1:3}\nonumber
  \sum_{n=1}^{N} & \omega_n \|Q_nx-x\|^2
  \geq \|Tx-x\|^2
  = \frac{\alpha(x)}{\lambda\sigma(x)}\tau(x)\|Tx-x\|^2\\
  & = \frac{m\alpha(x)}{(m+1)\lambda\sigma(x)}
  \left(\sum_{n=1}^{N} \omega_n \left\|Q_nx-x \right\|^2
  + \sum_{n=1}^{N} \omega_n \sum_{l=1}^{m_n}\|Q_{n,l}x-Q_{n,l-1}x\|^2\right),
\end{align}
which, after rearranging terms, leads us to the following estimate:
\begin{equation}\label{pr:th:ESA:s1:4}
  \sum_{n=1}^{N}\omega_n\|Q_nx-x\|^2 \geq \frac{m\alpha(x)}{(m+1)\lambda\sigma(x)-m\alpha(x)}
  \sum_{n=1}^{N} \omega_n \sum_{l=1}^{m_n}\|Q_{n,l}x-Q_{n,l-1}x\|^2.
\end{equation}
Using \eqref{pr:th:ESA:s1:2}, \eqref{pr:th:ESA:s1:4} and noticing that
\begin{equation}\label{pr:th:ESA:s1:5}
  \frac{1+m-m\alpha(x)}{m+1}\left(1+ \frac{m\alpha(x)}{(m+1)\lambda\sigma(x)-m\alpha(x)}\right)
  =  \frac{\Theta(x)}{\lambda \sigma(x)},
\end{equation}
we arrive at \eqref{eq:ESA:ineq1}.

\textit{Part (ii).} In order to show \eqref{eq:ESA:ineq2}, assume, in addition, that $x\in B(z,r)$. Given $i\in I$, choose $n\in\{1,\ldots,N\}$ and $1\leq p\leq m_n$, so that $i\in J_n$ and $Q_{n,p}=U_iQ_{n,p-1}$. Note that the operators $U_i$ map $B(z,r)$ into $B(z,r)$ since
  \begin{equation}\label{pr:th:ESA:s2:1}
    \|U_ix-z\| \leq \|x-z\| \leq r
  \end{equation}
  because each $U_i$ is quasi-nonexpansive. The same applies to each $Q_{n,l}$. Then, using the notation $C_i=\fix U_i$, we get

  \begin{align}\label{pr:th:ESA:s2:2}\nonumber
    d(x,C_i) & \leq \|x-P_{C_i}Q_{n,p-1}x\| \leq \|x-Q_{n,p-1}x\|+\|Q_{n,p-1}x-P_{C_i}Q_{n,p-1}\| \\
    & \leq \sum_{l=1}^{p-1} \|Q_{n,l}x-Q_{n,l-1}x\| + \frac{1}{\delta_{i}} \|U_iQ_{n,p-1}x-Q_{n,p-1}x\|
  \end{align}
  because $U_i$ is $\delta_i$-linearly regular over $B(z,r)$ and $\|Q_{n,p-1}x-P_{C_i}Q_{n,p-1}x\|= d(Q_{n,p-1}x, C_i)$. Since $\delta\leq \delta_i\leq1$ (compare with Remark \ref{rem:delta}) we obtain
  \begin{equation}
    d(x,C_i) \leq \frac{1}{\delta} \sum_{l=1}^{p} \|Q_{n,l}x-Q_{n,l-1}x\| \leq \sqrt{\frac{m_n}{\delta^2} \sum_{l=1}^{m_n} \|Q_{n,l}x-Q_{n,l-1}x\|^2},
  \end{equation}
  where the last inequality follows from the Cauchy-Schwarz inequality in $\mathbb{R}^{m_n}$. This shows that \eqref{eq:ESA:ineq2} holds true, which completes the proof of part (ii), because $i\in I$ was arbitrary.

  \textit{Part (iii).} Using \eqref{eq:ESA:ineq1}--\eqref{eq:ESA:ineq2}, we conclude that
  \begin{equation}\label{pr:th:ESA:s3:1}
    \Theta(x) \frac{\omega\delta^2}{m}\max_{i \in I}d(x,C_i)^2
    \leq \|x-z\|^2-\|T_{\lambda\sigma} x-z\|^2.
  \end{equation}
  Next, using the Cauchy-Schwarz inequality, we obtain
    \begin{align}\label{pr:th:ESA:s3:2}
    \nonumber
    \|T_{\lambda\sigma} x-z\|^2 &
    =\|T_{\lambda\sigma} x-x+x-z\|^2
    = \|T_{\lambda\sigma} x-x\|^2 - 2\langle T_{\lambda\sigma} x-x,x-z\rangle + \|x-z\|^2 \\
    \nonumber
    & \geq \|T_{\lambda\sigma} x-x\|^2 - 2\|T_{\lambda\sigma} x-x\|\|x-z\| + \|x-z\|^2 \\
    & \geq - 2\|T_{\lambda\sigma} x-x\|\|x-z\| + \|x-z\|^2,
  \end{align}
  which, when combined with \eqref{pr:th:ESA:s3:1}, and by setting $z=P_Cx$, leads to
  \begin{equation} \label{pr:th:ESA:s3:3}
    \Theta(x) \frac{\omega\delta^2}{2m}\max_{i \in I}d(x,C_i)^2
    \leq \|T_{\lambda\sigma} x-x\| d(x,C).
  \end{equation}
  On the other hand, since the family $\{C_i \mid i \in I\}$ is $\kappa$-linearly regular over $B(z,r)$, we obtain

  \begin{equation} \label{pr:th:ESA:s3:4}
  \frac 1 {\kappa^2} d^2(x,C)\leq \max_{i\in I} d^2(x,C_i),
  \end{equation}
  which, by \eqref{pr:th:ESA:s3:3}, proves \eqref{eq:ESA:ineq3}.

  In order to show \eqref{eq:ESA:ineq4}, it suffices to use \eqref{pr:th:ESA:s3:1} with $z=P_Cx$ and \eqref{pr:th:ESA:s3:4}, which leads to
  \begin{equation}\label{pr:th:ESA:s4:1}
    \Theta(x) \frac{\omega\delta^2}{m \kappa^2} d^2(x,C)
    \leq d^2(x,C).
  \end{equation}
  This completes the proof.
\end{proof}

\begin{remark}
  Observe that $\sigma\colon \mathcal H \to (0,\infty)$, defined by
  \begin{equation}
  \sigma(x):=
    \begin{cases}  \displaystyle
      \frac
      {\sum_{n=1}^{N} \omega_n \left\|Q_nx-x \right\|^2}
      {\left\|T x-x\right\|^{2}}
      & \text{if}\ x\notin C\\
      \ 1 & \text{if}\ x\in C,
    \end{cases}
    \end{equation}
  satisfies the inequalities $1\leq \sigma(x)\leq \tau(x)$.
\end{remark}

A direct application of Theorem \ref{th:ESA} to the extrapolated simultaneous cutter with $m=1$ leads to inequalities involving
\begin{equation}\label{}
  \Theta(x)=\lambda \sigma(x)\frac{2-\lambda \frac{\sigma(x)}{\tau(x)}}  {2-\frac{1}{\tau(x)}}
  \geq \frac 1 2 \lambda (2-\lambda).
\end{equation}
 Nevertheless, by slightly adjusting the proof of Theorem \ref{th:ESA}, we can replace the above $\Theta(x)$ by only $\lambda (2-\lambda)\sigma(x)$.

\begin{corollary}[Extrapolated Simultaneous Cutter]\label{th:ESC}
  Let $U_i\colon\mathcal{H}\to\mathcal{H}$ be a cutter, $i\in I:=\{1,\ldots,M\}$, such that $C:=\bigcap_{i\in I} \fix U_i \neq \emptyset$. Let $T$ be the simultaneous cutter operator defined by
  \begin{equation}\label{th:ESC:U}
    T:=\sum_{i=1}^{M} \omega_i U_i,
  \end{equation}
  where $\omega_i\in(0,1)$ with $\sum_{i=1}^{M}\omega_i=1$. Moreover, let $T_{\lambda\sigma}\colon\mathcal H \to \mathcal H$ be defined by
  \begin{equation}\label{eq:ESC:Usigma}
  T_{\lambda\sigma} x:= x + \lambda\sigma(x)(Tx-x),
  \end{equation}
  where the relaxation parameter $\lambda\in (0,2)$, the extrapolation functional $\sigma\colon\mathcal H \to [1,\infty)$ satisfies $1\leq \sigma(x) \leq \tau(x)$ for all $x\in\mathcal H$ and
  \begin{equation}\label{eq:ESC:tau}
    \tau(x):=
    \begin{cases}  \displaystyle
      \frac
      {
      \sum_{i=1}^{M} \omega_i \left\|U_ix-x \right\|^2
      }
      {
      \left\|T x-x\right\|^{2}
      } & \text{if}\ x\notin C\\
      \ 1 & \text{otherwise}.
    \end{cases}
  \end{equation}
  Then the following statements hold:
  \begin{enumerate}[(i)]
  \item $T_{\lambda\sigma}$ is $\lambda(2-\lambda)$-SQNE, where $\fix T_{\lambda\sigma}=C$. Moreover, for all $x\in \mathcal H$ and $z\in C$, we have
      \begin{equation}\label{eq:ESC:ineq1}
        \lambda(2-\lambda)\sigma (x)\sum_{i=1}^{M}\omega_i \|U_ix-x\|^2
        \leq  \|x-z\|^2 - \|T_{\lambda\sigma} x- z\|^2.
      \end{equation}
  \item If for every $i\in I$, the operator $U_i$ is $\delta_i$-linearly regular, $\delta_i\in(0,1]$, on the ball $B(z,r)$ for some $z\in C$ and some $r>0$, then for all $x\in B(z,r)$, we have
      \begin{equation}\label{eq:ESC:ineq2}
        \delta^2 \max_{i\in I} d^2(x,\fix U_i) \leq \sum_{i=1}^{M} \|U_ix-x\|^2,
      \end{equation}
      where $\delta:=\min_{i\in I} \delta_i$.

  \item If, in addition, the family $\{\fix U_i\mid i\in I\}$ is $\kappa$-linearly regular over $B(z,r)$, then $T_{\lambda\sigma}$ also satisfies
    \begin{equation}\label{eq:ESC:ineq3}
      \lambda(2-\lambda)\sigma(x) \frac{\omega\delta^2}{2\kappa^2 }  d\big(x,C\big) \leq
      \|T_{\lambda\sigma} x-x\|
    \end{equation}
    for all $x\in B(z,r)$, where $\omega:=\min_{i\in I}\omega_i$. Moreover, in this case
		\begin{equation}\label{eq:ESC:ineq4}
		\sigma(x) \leq \frac{\kappa^2}{\lambda(2-\lambda)\omega\delta^2}.
		\end{equation}
  \end{enumerate}
\end{corollary}

\begin{proof}
  Following the proof of Theorem \ref{th:ESA}, one can observe that the most significant change corresponds to \eqref{pr:th:ESA:s1:2}, which takes the following form:
  \begin{align}\label{pr:th:ESC:1}\nonumber
    \|T_{\lambda\sigma} x-z\|^2 & \leq \|x-z\|^2
    -\lambda\sigma(x)\big(2-\alpha(x)\big) \tau (x)\|T x- x\|^2 \\
    & \leq \|x-z\|^2
    -\lambda(2-\lambda)\sigma(x) \sum_{i=1}^M\omega_i\|U_ix-x\|^2.
  \end{align}
  Thus we have shown \eqref{eq:ESC:ineq1}. Inequality \eqref{eq:ESC:ineq2} follows trivially from the linear regularity of each $U_i$. By combining \eqref{eq:ESC:ineq1} and \eqref{eq:ESC:ineq2}, we arrive at
  \begin{equation}\label{pr:th:ESC:2}
    \lambda(2-\lambda) \sigma(x) \omega\delta^2\max_{i\in I}d(x,C_i)^2
    \leq \|x-z\|^2-\|T_{\lambda\sigma} x-z\|^2.
  \end{equation}
  The rest of the proof remains the same as in the proof of Theorem \ref{th:ESA} with \eqref{pr:th:ESA:s3:1} replaced by \eqref{pr:th:ESC:2}.
\end{proof}

\begin{corollary}[Extrapolated Cyclic Cutter]\label{th:ECC}
  Let $U_i\colon\mathcal{H}\to\mathcal{H}$ be a cutter, $i\in I:=\{1,\ldots,M\}$, such that $C:=\bigcap_{i\in I} \fix U_i \neq \emptyset$. Let $T$ be the cyclic cutter operator defined by
  \begin{equation}\label{th:ECC:U}
    T:=U_M\ldots U_1.
  \end{equation}
  Moreover, let $T_{\lambda\sigma}\colon\mathcal H \to \mathcal H$ be defined by
  \begin{equation}\label{eq:ECC:Usigma}
  T_{\lambda\sigma} x:= x + \lambda\sigma(x)(Tx-x),
  \end{equation}
  where the relaxation parameter $\lambda\in (0,1+\frac 1 M)$, the extrapolation functional $\sigma\colon\mathcal H \to [1,\infty)$ satisfies $1\leq \sigma(x) \leq \tau(x)$ for all $x\in\mathcal H$ and
  \begin{equation}\label{eq:ECC:tau}
    \tau(x):=
    \begin{cases}  \displaystyle
      \frac{M}{M+1} \left( 1+
      \frac
      {
      \sum_{i=1}^{M} \left\|U_i\ldots U_1 x-U_{i-1}\ldots U_1 x \right\|^2
      }
      {
      \left\|T x-x\right\|^{2}
      }
      \right)
      & \text{if}\ x\notin C\\
      \ 1 & \text{otherwise},
    \end{cases}
  \end{equation}
  where $U_{i-1}\ldots U_1 x:=x$ for $i=1$. Then the following statements hold:
  \begin{enumerate}[(i)]
  \item $T_{\lambda\sigma}$ is $\rho$-SQNE, where $\fix T_{\lambda\sigma}=C$ and
  \begin{equation}\label{eq:ECC:rho}
    \rho:=\frac 1 {\lambda M} + \frac{1-\lambda}{\lambda}>0.
  \end{equation}
  Moreover, for all $x\in \mathcal H$ and $z\in C$, we have
    \begin{equation}\label{eq:ECC:ineq1}
      \Theta(x)\sum_{i=1}^{M} \left\|U_i\ldots U_1 x-U_{i-1}\ldots U_1 x \right\|^2
      \leq  \|x-z\|^2 - \|T_{\lambda\sigma} x- z\|^2,
    \end{equation}
  where
    \begin{equation}\label{eq:ECC:Delta}
      \Theta(x):=
      \lambda\sigma (x)
      \frac{1+M-M\lambda\frac{\sigma(x)}{\tau(x)}}
           {1+M-\frac{M}{\tau(x)}}
      \geq \lambda \left(1-\lambda \frac {M} {1+M}\right)>0.
    \end{equation}
  \item If for every $i\in I$, the operator $U_i$ is $\delta_i$-linearly regular, $\delta_i\in(0,1]$, on the ball $B(z,r)$ for some $z\in C$ and some $r>0$, then for all $x\in B(z,r)$, we have
      \begin{equation}\label{eq:ECC:ineq2}
        \frac{\delta^2}M \max_{i\in I} d^2(x,\fix U_i)
        \leq \sum_{i=1}^{M} \left\|U_i\ldots U_1 x-U_{i-1}\ldots U_1 x \right\|^2,
      \end{equation}
      where $\delta:=\min_{i\in I} \delta_i$.

  \item If, in addition, the family $\{\fix U_i\mid i\in I\}$ is $\kappa$-linearly regular over $B(z,r)$, then $T_{\lambda\sigma}$ also satisfies
    \begin{equation}\label{eq:ECC:ineq3}
      \Theta(x) \frac{\delta^2}{2M\kappa^2 }  d\big(x,C\big) \leq
      \|T_{\lambda\sigma} x-x\|
    \end{equation}
    for all $x\in B(z,R)$. Moreover, in this case,
		\begin{equation}\label{eq:ECC:ineq4}
		\Theta(x) \leq \frac{M\kappa^2}{\delta^2}.
		\end{equation}
  \end{enumerate}
\end{corollary}
\begin{proof}
  Apply Theorem \ref{th:ESA} with one string $J=(1,\ldots,M)$ and $m=M$.
\end{proof}

\begin{remark}[Non-extrapolated operators]\label{rem:LR}
  In the case of the non-extrapolated operator, where $\sigma(x)=1$ for all $x \in\mathcal H$, we can simplify the lower bound for $\Theta(x)$ while assuming that $\lambda \leq 1$. Indeed, we have $\Theta(x)\geq \lambda$, which coincides with the additional estimate made for the simultaneous cutter presented in Corollary \ref{th:ESC}. Moreover, by substituting $\lambda=1$, we see that $\Theta(x)=1$ for all $x\in\mathcal H$. This may simplify some of the estimates related to the linear regularity (LR) of operators and thus influence the linear rate of convergence of some iterative methods.

  In particular, by Theorem \ref{th:ESA} (iii), the string averaging cutter $T$ defined in \eqref{th:ESA:U} is linearly regular (LR) over the ball $B(z,r)$ and satisfies
  \begin{equation}\label{eq:rem:LR}
    \frac{\omega\delta^2}{2m\kappa^2}  d\big(x,C\big)
      \leq \|T x-x\|
  \end{equation}
  for all $x \in B(z,r)$. The above inequality coincides with \cite[Lemma 8]{BargetzReichZalas2017}. We recall that \cite[Lemma 8]{BargetzReichZalas2017} was established for metric projections only ($U_i=P_{C_i}$) which are 1-LR and therefore $\delta=1$ in \eqref{eq:rem:LR}.

  On the other hand, the LR of the operator $T$ based on general $U_i$'s can be found in \cite[Example 5.7]{CegielskiReichZalas2018}. This result follows from \cite[Corollaries 5.3 and 5.6]{CegielskiReichZalas2018}, where the authors present LR estimates for the non-extrapolated convex combination and product of operators that coincide with \eqref{eq:ESC:ineq3} and \eqref{eq:ECC:ineq3}, respectively. However, this result requires the additional assumption that each subfamily $\{\fix U_j\mid j\in J_n\}$ is LR over $B(z,r)$, $n=1,\ldots, N$. Indeed, by  \cite[Corollary 5.6]{CegielskiReichZalas2018}, the product $Q_n$ is $\delta^2/(2m_n\kappa_n^2)$-LR over $B(z,r)$ due to the $\kappa_n$-LR of the family $\{\fix U_j\mid j\in J_n\}$ over $B(z,r)$. This, by \cite[Corollary 5.3]{CegielskiReichZalas2018} and the assumption that the family $\{\fix U_i\mid i\in I\}$ is $\kappa$-LR over $B(z,r)$, implies that the operator $T$ defined as a convex combination is $\delta_T$-LR over $B(z,r)$, where
  \begin{equation}\label{}
    \delta_T=
    \frac
      {\omega \left(\frac{\delta^2} {2 m \max \kappa_n^2}\right)^2}
      {(2\kappa^2)} > \frac{\omega\delta^2}{2m\kappa^2}.
  \end{equation}
  This reasoning cannot be applied in every case in view of the counterexample provided in \cite{ReichZaslavski2014} according to which it may happen that no subfamily of an LR family of sets is LR. We emphasize at this point that the argument we presented in the proof of Theorem \ref{th:ESA} does not require any additional regularity of subfamilies.
\end{remark}

\section{Extrapolated Iterative Methods}\label{sec:methods}

\begin{theorem}[Extrapolated Dynamic String Averaging Method]\label{th:main}
  Let $U_i\colon\mathcal{H}\to\mathcal{H}$ be a cutter, $i\in I:=\{1,\ldots,M\}$, such that $C:=\bigcap_{i\in I} \fix U_i \neq \emptyset$. Let the sequence $\{x^k\}_{k=0}^{\infty}$ be defined by the following method:
  \begin{equation}\label{th:main:xk}
  x^0\in\mathcal{H} \qquad \text{and} \qquad x^{k+1} := x^k+\lambda_k \sigma_k(x^k)\left(T_k x^{k}-x^k\right),
  \quad k=0,1,2,\ldots,
  \end{equation}
  where $T_k:=\sum_{n=1}^{N_k} \omega_n^k Q_n^k$ is the string averaging cutter operator, $Q_n^k:=\prod_{j\in J_n^k} U_j$ is a product of operators along the string $J_n^k\subseteq I$, $\omega_n^k >0$ satisfy $\sum_{n=1}^{N_k}\omega_n^{k}=1$, $\lambda_k>0 $ and $\sigma_k\colon\mathcal H \to [1,\infty)$.

  Assume that $\omega_n^k\in[\omega,1]$ for some $\omega \in (0,1]$, $\lambda_k \in[\underline \lambda, \overline \lambda]$ for some $\underline{\lambda} \in (0,1]$ and $\overline \lambda \in [1,1+\frac 1 m)$, $\sigma_k(x)\in [1,\tau_k(x)]$ for each $x\in \mathcal H$, where
  \begin{equation}\label{th:main:tau}
    \tau_k(x):=
    \begin{cases}  \displaystyle
      \frac
      {
      \frac{m_k}{m_k+1}\sum\limits_{n=1}^{N} \omega_n^k
        \left(
        \left\|Q_n^kx-x \right\|^2
        + \sum\limits_{l=1}^{m_n^k}\|Q_{n,l}^kx-Q_{n,l-1}^{k}x\|^2 \right)
      }
      {
      \left\|T_k x-x\right\|^{2}
      } & \text{if}\ x\notin \bigcap\limits_{i\in I_k}\fix U_i\\
      \ 1 & \text{otherwise,}
    \end{cases}
  \end{equation}
  $m_n^k:=|J_n^k|$, $m_k:=\max_{1\leq n\leq N_k} m_n^k$, $ m:=\sup_k m_k<\infty$ and $Q_{n,l}^k$ is a product of the first $l$ operators along the string $J_n^k$, $Q_{n,0}^k:=\id$. Moreover, assume that the control $\{I_k\}_{k=0}^\infty$ is $s$-intermittent, that is, there is an integer $s\geq 1$ such that $I = I_k\cup\ldots\cup I_{k+s-1}$ for each $k=0,1,2,\ldots$, where $I_k := \bigcup_{n=1}^{N_k} J_n^k$. We set $r:=d(x^0,C)$ and $B:=B(P_Cx^0, r)$.
  Then the following statements are true:
  \begin{enumerate}[(i)]
  \item If each $U_i$ is weakly regular over $B$, then $\{x^k\}_{k=0}^{\infty}$ converges weakly to some point $x^\infty\in C \cap B$.
  \item If each $U_i$ is regular over $B$ and the family $\{\fix U_i \mid i\in I\}$ is regular over $B$, then the convergence to $x^\infty$ is in norm.
  \item If each $U_i$ is $\delta_i$-linearly regular over $B$ and the family $\{\fix U_i\mid i\in I\}$ is $\kappa$-linearly regular over $B$, then, for all $k=0,1,2,\ldots$, we have
       \begin{equation}\label{th:main:equiv}
        \frac{\delta}{2\kappa}\|x^k-x^\infty\|
        \leq \max_{i\in I}\|U_ix^k-x^k\|
        \leq \|x^k-x^\infty\|,
      \end{equation}
      and
      \begin{equation}\label{th:main:rate}
        \|x^k-x^\infty\|
        \leq 2d(x^0,C)
        \left(
        \sqrt[\leftroot{-1}\uproot{2}\large \displaystyle 2s]
        {
        1-\frac{\omega\rho\delta^2}{s\kappa^2 }
        \min \left\{ \frac 1 {\omega\delta^2},\ \inf_k \Theta_{k} \right\}
        }
        \right)^{k-s+1},
      \end{equation}
      where $\delta := \min_{i\in I}\delta_i$, $\rho:=\frac 1 {m\overline{\lambda }} + \frac{1-\overline{\lambda}}{\overline{\lambda }}$ and
      \begin{equation}\label{th:main:parameters}
        \Theta_k:=
          \lambda_k\sigma_k(x^k)
          \frac{1+m_k-m_k\lambda_k\frac{\sigma_k(x^k)}{\tau_k(x^k)}}
             {1+m_k-\frac{m_k}{\tau_k(x^k)}};
      \end{equation}
      compare with \eqref{eq:ESA:Theta}. In particular, the convergence to $x^\infty$ is $R$-linear.
  \end{enumerate}
\end{theorem}

\begin{proof}
  We show that $\{x^k\}_{k=0}^\infty$ is Fej\'{e}r monotone with respect to the set $C$. We use the notation
  \begin{equation}\label{th:main:notation}
    T^k_{\lambda_k\sigma_k}:=\id + \lambda_k\sigma_k(\cdot) \left(T_k -\id\right);
  \end{equation}
  compare with \eqref{eq:ESA:Usigma}. The iterative step of \eqref{th:main:xk} can be rewritten as $x^{k+1}:=T^k_{\lambda_k\sigma_k}x^k$. By Theorem \ref{th:ESA}, each $T^k_{\lambda_k\sigma_k}$ is $\rho_k$-SQNE and $\fix T^k_{\lambda_k\sigma_k}=\fix T_k =\bigcap_{i\in I_k}C_i$, where we set $C_i := \fix U_i$ and
  \begin{equation}\label{pr:rho}
    \rho_k:=\frac{1}{\lambda_k m_k} + \frac{1-\lambda_k}{\lambda_k}
    \geq \rho >0.
  \end{equation}
  Hence, for each $k=0,1,2,\ldots,$ and $z\in C$, we have
	\begin{equation}\label{eq:fejermonotone}
	 \|x^{k+1}-z\|^2=\|T^k_{\lambda_k\sigma_k} x^k-z\|^2
   \leq \|x^k-z\|^2-\rho\|T^k_{\lambda_k\sigma_k} x^k-x^k\|^2\leq\|x^k-z\|^2.
	\end{equation}
  Due to the above inequality, we see that the sequence $\{x^k\}_{k=0}^\infty \subseteq B$ is Fej\'{e}r monotone. In particular, since the sequence $\{\|x^k-z\|\}_{k=0}^\infty$ is bounded and decreasing, it is convergent and
  \begin{equation}\label{eq:asymptoticregularity}
    \|T_k x^k-x^k\|
    = \frac 1 {\lambda_k \sigma_k(x^k)} \|T^k_{\lambda_k\sigma_k} x^k-x^k\|
    \leq \frac 1 {\underline\lambda} \|T^k_{\lambda_k\sigma_k} x^k-x^k\|
    = \frac 1 {\underline\lambda}\|x^{k+1}-x^k\|\rightarrow 0,
  \end{equation}
  where the inequality holds because $\lambda_k\sigma_k(x^k)\geq \underline\lambda$.

  \textit{Part (i).} Assume that for each $i\in I$, the operator $U_i$ is weakly regular over $B$. Let $i\in I$ and let $x$ be an arbitrary cluster point of $\{x^k\}_{k=0}^\infty$. By Theorem \ref{th:Fejer} (i), it suffices to show that $x\in \fix U_i$, which by the arbitrariness of $i$, will imply that $x\in C$. The key step in this part of the proof is to apply a variant of \cite[Lemma 3.4]{ReichZalas2016} to a certain subsequence of $\{T_k\}_{k=0}^\infty$ and $\{x^k\}_{k=0}^\infty$. We note here that \cite[Lemma 3.4]{ReichZalas2016} was established under the assumption that each $U_i$ is weakly regular (see ``Opial's demi-closedness principle'' in \cite{ReichZalas2016}). Nonetheless, the result itself holds when restricted to the operators $U_i$ which are weakly regular only over the ball $B$.

  To see this, let $x^{n_k}\rightharpoonup x$. By the assumption that the control $I_k$ is $s$-intermittent, for each $k=0,1,2,\ldots,$ there is $l_k\subseteq\{0,1,\ldots,s-1\}$ such that $i\in I_{n_k+l_k}$. This implies that $\fix T_{n_k+l_k}\subseteq\fix U_i$. Using (\ref{eq:asymptoticregularity}), we deduce that
	\begin{equation}
	  x^{n_k+l_k}\rightharpoonup x \qquad \text{and} \qquad
    \|T_{n_k+l_k}x^{n_k+l_k}-x^{n_k+l_k}\|\rightarrow 0.
	\end{equation}
  By applying \cite[Lemma 3.4]{ReichZalas2016} to $\{x^{n_k+l_k}\}_{k=0}^\infty$ and $\{T_{n_k+l_k}\}_{k=0}^\infty$, we see that $x\in\fix U_i$ which, as explained above, completes this part of the proof.

  \textit{Part (ii).} Assume that the family $\{C_i\mid i\in I\}$ and each operator $U_i$, $i\in I$, are regular over $B$. Let $i\in I$. By Theorem \ref{th:Fejer} (ii), it is enough to show that $\lim_{k\rightarrow\infty} d(x^k,C_i) =0$, which by the regularity of the family of sets will lead to $\lim_{k\rightarrow\infty}d(x^k,C) =0$. Similarly as in (i), we apply a variant of \cite[Lemma 3.5]{ReichZalas2016} which was established for boundedly regular operators (called there ``approximately shrinking''). We emphasize that this result holds too, when restricted to the operators $U_i$ which are regular only over the ball $B$.

  As in the previous part, we deduce that there is a sequence $l_k\subseteq\{0,1,\ldots,s-1\}$ such that $i\in I_{k+l_k}$. Again, by using (\ref{eq:asymptoticregularity}), we have
	\begin{equation}
	  x^{k+l_k}-x^k\rightarrow 0 \qquad \text{and} \qquad
    \|T_{k+l_k}x^{k+l_k}-x^{k+l_k}\|\rightarrow 0.
	\end{equation}
   Applying \cite[Lemma 3.5]{ReichZalas2016} to $\{x^{k+l_k}\}_{k=0}^\infty$ and $\{T_{k+l_k}\}_{k=0}^\infty$, we see that
	\begin{equation}
	\lim_{k\rightarrow\infty}d(x^{k+l_k},C_i) =0.
	\end{equation}
	Using the properties of projections and the triangle inequality, we obtain
	\begin{equation}
    d(x^k,C_i)= \|x^k-P_{C_i}x^k\|\leq\|x^k-P_{C_i}x^{k+l_k}\|
    \leq \|x^k-x^{k+l_k}\|+\|x^{k+l_k}-P_{C_i}x^{k+l_k}\|,
	\end{equation}
	which, by (\ref{eq:asymptoticregularity}), implies that $\lim_{k\rightarrow\infty} d(x^k,C_i) =0$. So this part of the proof is complete.

  \textit{Part (iii).} We divide the remaining part of the proof into three steps.

  \textit{Step 1.} We first show that \eqref{th:main:equiv} holds. Indeed, since $x^k\to x^\infty$, we can use Theorem \ref{th:Fejer} (v), the facts that the family $\{C_i\mid i\in I\}$ is $\kappa$-linearly regular and that each operator $U_i$, $i\in I$, is $\delta$-linearly regular to arrive at
  \begin{equation}
    \|x^k-x^\infty\| \leq 2 d(x^k,C) \leq 2\kappa \max_{i\in I} d(x^k, C_i) \leq \frac{2\kappa}{\delta} \max_{i\in I} \|U_ix^k-x^k\|.
  \end{equation}
  On the other hand, using the facts that $U_i$ is a cutter, $C\subseteq C_i$, $x^\infty\in C$, we get
  \begin{equation}
    \max_{i\in I} \|U_ix^k-x^k\| \leq \max_{i\in I} d(x^k,C_i) \leq d(x^k, C) \leq \|x^k-x^\infty\|,
  \end{equation}
  which yields \eqref{th:main:equiv}.

  \textit{Step 2.}
  We show that the inequality
  \begin{equation}\label{eq:Step1}
    \max_{i\in I}d(x^{ks}, C_i)^2 \leq
    \frac{s}{\omega\rho\delta^2} \cdot
    \max_{l=0,\ldots,s-1}\frac{1}{\theta_{ks+l} }
    \left(\|x^{ks}-z\|^2-\|x^{(k+1)s}-z\|^2\right)
  \end{equation}
  holds for every $k=0,1,2,\ldots$ and every $z\in C$, where $\theta_k:=\min\{\frac 1 {\omega \delta^2},\ \Theta_k\}$. To this end, fix $i\in I$ and $z\in C$. Given an integer $k$, we choose $l_k\in\{0,\ldots,s-1\}$ to be the smallest index so that $i\in I_{ks+l_k}$. By the definition of the metric projection and by the triangle inequality, we have
  \begin{equation}\label{eq:DistNonExp}
    d(x^{ks},C_i)\leq \|x^{ks}-x^{ks+l_k}\| + d(x^{ks+l_k},C_i).
  \end{equation}
  Using Theorem \ref{th:ESA} (i)-(ii) applied to $T_{ks+l_k}$ and $\lambda_{ks+l_k}\sigma_{ks+l_k}$ (see \eqref{pr:th:ESA:s3:1}), and the inequalities $\Theta_k\geq \theta_k$ and $\frac 1 m \geq \rho$, we get
  \begin{equation}\label{}
    d(x^{ks+l_k}, C_i)^2 \leq
    \frac{1}{\omega \rho \delta^2 \theta_{ks+l_k} }
    \left(\|x^{ks+l_k}-z\|^2-\|x^{ks+l_k+1}-z\|^2\right).
  \end{equation}
  Consequently, using the Cauchy-Schwarz inequality in $\mathbb{R}^{l_k}$, we obtain
  \begin{align}\nonumber
    d(x^{ks},C_i)^2
    & \leq \left(\sum_{p=ks+1}^{ks+l_k}\|x^{p}-x^{p-1}\| + d(x^{ks+l_k},C_i)\right)^2
    \\ \nonumber
    & \leq (l_k+1) \left(\sum_{p=ks+1}^{ks+l_k}\|x^{p}-x^{p-1}\|^2 + d(x^{ks+l_k},C_i)^2\right)
    \\
    & \leq s \left(\sum_{p=ks+1}^{ks+l_k}\|x^{p}-x^{p-1}\|^2 + \frac{1}{\omega\rho\delta^2 \theta_{ks+l_k}} \left(\|x^{ks+l_k}-z\|^2-\|x^{ks+l_k+1}-z\|^2\right)\right).
  \end{align}
  Since $T^{p-1}_{\lambda_{p-1} \sigma_{p-1}}$ is $\rho$-SQNE, we get
  \begin{equation}
    \|x^p-x^{p-1}\|^2 \leq \frac 1 \rho (\|x^{p-1}-z\|^2-\|x^p-z\|^2),
  \end{equation}
  and therefore, since $ \delta^2 \omega \theta_{ks+l_k}\leq 1$,
  \begin{align}
    \nonumber
    d(x^{ks},C_i)^2
    & \leq \frac{s}{\omega\rho\delta^2  \theta_{ks+l_k}} \left( \|x^{ks}-z\|^2-\|x^{ks+l_k+1}-z\|^2\right)
    \\
    &  \leq \frac{s}{\omega\rho\delta^2  \theta_{ks+l_k}} \left(\|x^{ks}-z\|^2-\|x^{k(s+1)}-z\|^2\right),
  \end{align}
  where the last inequality is a consequence of the Fej\'{e}r monotonicity of $\{x^k\}_{k=0}^{\infty}$. The above inequality yields \eqref{eq:Step1}.

  \textit{Step 3.} Setting $z=P_{C}x^{ks}$ in~\eqref{eq:Step1} and using the inequality
  \begin{equation}
    \|x^{(k+1)s}-P_{C}x^{(k+1)s}\|\leq  \|x^{(k+1)s}-P_{C}x^{ks}\|,
  \end{equation}
  we deduce that
  \begin{equation}\label{eq:pr:main:s2}
    \max_{i\in I}d(x^{ks}, C_i)^2 \leq
    \frac{s}{\omega\rho\delta^2 \min\limits_{l=0,\ldots,s-1}\theta_{ks+l}}
    \left(d(x^{ks},C)^2-d(x^{(k+1)s},C)^2\right).
  \end{equation}
  Using the linear regularity of $\{C_i\mid i\in I\}$ on $B(P_Cx^0,r)$, we get
  \begin{equation}
    d(x^{ks},C) \leq \kappa \max_{i\in I} d(x^{ks},C_i)
  \end{equation}
  which, when combined with \eqref{eq:pr:main:s2}, leads to
  \begin{equation}\label{eq:pr:main:s2b}
    d(x^{(k+1)s},C)^2 \leq \left(1-\frac{\omega\rho \delta^2}{s \kappa^2}
    \cdot \min_{l=0,\ldots,s-1}\theta_{ts+l} \right) d(x^{ks},C)^2
  \end{equation}
  for all $k=0,1,2,\ldots$. Using the Fej\'{e}r monotonicity of $\{x^k\}_{k=0}^\infty$, Theorem \ref{th:Fejer} (v) and \eqref{eq:pr:main:s2b}, we arrive at
  \begin{align}\label{eq:pr:main:s2c} \nonumber
    \|x^{k}-x^\infty\|
    & \leq \| x^{\lfloor k/s\rfloor s}-x^\infty\|
    \leq 2 d(x^{\lfloor k/s\rfloor s},C)
    \\ \nonumber
    &\leq 2 d (x^0,C) \prod_{t=0}^{ \lfloor k/s\rfloor -1 }
        \sqrt
        {
        1-\frac{\omega\rho\delta^2}{s\kappa^2 } \cdot
        \min_{l=0,\ldots,s-1}\theta_{ts+l}
        }
    \\
    & \leq 2 d (x^0,C)
    \left(
      \sqrt[\leftroot{-1}\uproot{2}\large \displaystyle 2s]
        {
        1-\frac{\omega\rho\delta^2}{s\kappa^2 }
        \min \left\{ \frac 1 {\omega\delta^2},\ \inf_k \Theta_{k} \right\}
        }
    \right)^{\lfloor k/s\rfloor s}
  \end{align}
  which holds for all $k=s-1,s,s+1,\ldots$. Observe that the above inequality holds true also for $k=0,1,2,\ldots,s$. Since $\lfloor k/s\rfloor s \geq k-s+1$, we have established \eqref{th:main:rate}. This completes the proof.
\end{proof}

\begin{corollary}[Extrapolated Simultaneous Cutter Method]\label{th:ESCM}
  Let $U_i\colon\mathcal{H}\to\mathcal{H}$ be a cutter, $i\in I:=\{1,\ldots,M\}$, such that $C:=\bigcap_{i\in I} \fix U_i \neq \emptyset$. Let the sequence $\{x^k\}_{k=0}^{\infty}$ be defined by the following method:
  \begin{equation}\label{th:ESCM:xk}
  x^0\in\mathcal{H} \qquad \text{and} \qquad x^{k+1} := x^k+\lambda_k\sigma_k(x^k)\left(\sum_{i\in I_k} \omega_i U_{i} x^{k}-x^k\right),
  \quad k=0,1,2,\ldots,
  \end{equation}
  where the control set of indices $I_k\subseteq I$ is nonempty,
  $\omega_n^k >0$ satisfy $\sum_{n=1}^{N_k}\omega_n^{k}=1$, $\lambda_k>0 $ and $\sigma_k\colon\mathcal H \to [1,\infty)$.

  Assume that $\omega_n^k\in[\omega,1]$ for some $\omega \in (0,1]$, $\lambda_k \in[\underline \lambda, \overline \lambda]$ for some $\underline{\lambda} \in (0,1]$ and $\overline \lambda \in [1,2)$, $\sigma_k(x)\in [1,\tau_k(x)]$ for each $x\in \mathcal H$, where

  \begin{equation}\label{th:ESCM:tau}
    \tau(x):=
    \begin{cases} \displaystyle
      \frac
      {\sum_{i\in I_k} \omega_i^k \left\|U_ix-x\right\|^2}
      {\left\|\sum_{i\in I_k} \omega_i^k U_i x-x\right\|^{2}}
      & \text{if}\ x\notin \bigcap\limits_{i\in I_k}\fix U_i \\
      \ 1 & \text{otherwise.}
    \end{cases}
  \end{equation}
  Moreover, assume that the control $\{I_k\}_{k=0}^\infty$ is $s$-intermittent for some $s\geq 1$. As in Theorem \ref{th:main}, we set $r:=d(x^0,C)$ and $B:=B(P_Cx^0,r)$. Then the following statements hold:
  \begin{enumerate}[(i)]
  \item If each $U_i$ is weakly regular over $B$, then $\{x^k\}_{k=0}^{\infty}$ converges weakly to some point $x^\infty\in C \cap B$.
  \item If each $U_i$ is regular over $B$ and the family $\{\fix U_i \mid i\in I\}$ is regular over $B$, then the convergence to $x^\infty$ is in norm.
  \item If each $U_i$ is $\delta_i$-linearly regular over $B$ and the family $\{\fix U_i\mid i \in I\}$ is $\kappa$-linearly regular over $B$, then, in addition to \eqref{th:main:equiv}, we have
      \begin{equation}\label{th:ESCM:tau}
        \|x^k - x^\infty\|
        \leq 2d(x^0,C)
        \left(
          \sqrt[\leftroot{-1}\uproot{2}\large \displaystyle 2s]
          {1-\frac{\omega\rho\delta^2}{s\kappa^2}
          \min\left\{ \frac 1 {\omega\delta^2},\ \inf_k  \lambda_k(2-\lambda_k)\sigma_k(x^k) \right\} }
        \right)^{k-s+1}
      \end{equation}
      for all $k=0,1,2,\ldots$, where $\delta:=\min_i\delta_i$ and $\rho:=\frac{2-\overline{\lambda}}{\overline{\lambda}}$.
  \end{enumerate}
\end{corollary}
\begin{proof}
  The proof is the same as the proof of Theorem \ref{th:main}, where instead of Theorem \ref{th:ESA} we use Corollary \ref{th:ESC} in order to define $\theta_k:=\min\{\frac{1}{\omega\delta^2},\ \lambda_k(2-\lambda_k)\sigma_k(x^k)\}$.
\end{proof}

\begin{corollary}  [Extrapolated Cyclic Cutter Method] \label{th:ECCM}
  Let $U_i\colon\mathcal{H}\to\mathcal{H}$ be a cutter, $i\in I:=\{1,\ldots,M\}$, such that $C:=\bigcap_{i\in I} \fix U_i \neq \emptyset$. Let the sequence $\{x^k\}_{k=0}^{\infty}$ be defined by the following method:
  \begin{equation}\label{th:ECCM:xk}
  x^0\in\mathcal{H} \qquad \text{and} \qquad x^{k+1} := x^k+\lambda_k\sigma_k(x^k)\left(U_M\ldots U_1 x^{k}-x^k\right),
  \quad k=0,1,2,\ldots,
  \end{equation}
  where $\lambda_k>0 $ and $\sigma_k\colon\mathcal H \to [1,\infty)$.

  Assume that $\lambda_k \in[\underline \lambda, \overline \lambda]$ for some $\underline{\lambda} \in (0,1]$ and $\overline \lambda \in [1,1+\frac 1 M)$ and $\sigma_k(x)\in [1,\tau_k(x)]$ for each $x\in \mathcal H$, where
  \begin{equation}\label{th:ECCM:tau}
    \tau(x):=
    \begin{cases} \displaystyle
      \frac{M}{M+1}\left( 1+
      \frac
      {\sum_{i=1}^M \left\|U_i\ldots U_1 x-U_{i-1}\ldots U_1 x\right\|^2}
      {\left\|U_M\ldots U_1 x-x\right\|^{2}}
      \right)
      & \text{if}\ x\notin C\\
      \ 1 & \text{otherwise.}
    \end{cases}
  \end{equation}
  As in Theorem \ref{th:main}, we set $r:=d(x^0,C)$ and $B:=B(P_Cx^0,r)$.
  Then the following statements hold:
  \begin{enumerate}[(i)]
  \item If each $U_i$ is weakly regular over $B$, then $\{x^k\}_{k=0}^{\infty}$ converges weakly to some point $x^\infty\in C \cap B$.
  \item If each $U_i$ is regular over $B$ and the family $\{\fix U_i \mid i\in I\}$ is regular over $B$, then the convergence to $x^\infty$ is in norm.
  \item If each $U_i$ is $\delta_i$-linearly regular over $B$ and the family $\{\fix U_i\mid i \in I\}$ is $\kappa$-linearly regular over $B$, then, in addition to \eqref{th:main:equiv}, we have
      \begin{equation}\label{th:ECCM:tau}
        \|x^k - x^\infty\|
        \leq 2d(x^0,C)
        \left(
          \sqrt
          {1-\frac{\rho\delta^2}{M\kappa^2}
          \min\left\{ \frac 1 {\delta^2},\ \inf_k  \Theta_k \right\} }
        \right)^{k}
      \end{equation}

      for all $k=0,1,2,\ldots$, where $\delta:= \min_{i\in I}\delta_i$, $\rho:=\frac 1 {M\overline{\lambda }} + \frac{1-\overline{\lambda}}{\overline{\lambda }}$ and
      \begin{equation}\label{th:ECCM:Theta}
        \Theta_k := \lambda_k\sigma_k(x^k)
          \frac{1+M-M\lambda_k\frac{\sigma_k(x^k)}{\tau(x^k)}}
             {1+M-\frac{M}{\tau(x^k)}}.
      \end{equation}
  \end{enumerate}
\end{corollary}

\begin{remark}\label{rem:lambda}
  Allowing $\overline \lambda \in [1,1+\frac 1 m )$, enables us to capture examples presented in the introduction within the framework of Theorem \ref{th:main}. Indeed, let us consider, for example,  the iterative scheme \eqref{th:main:xk} with $\lambda_k>0$ and $\sigma_k(x^k) >0$. Following \cite{NikazadMirzapour2017}, assume that $\sigma_k$ is defined as in \eqref{intro:sigma:SA2} and $\lambda_k\in [\varepsilon,2-\varepsilon]$ for some $\varepsilon\in (0,1)$. Then we have
  \begin{equation}\label{}
    \lambda_k \sigma_k(x^k)=\frac {m+1} {2m} \lambda_k \tau_k(x^k),
  \end{equation}
  where $\tau_k$ is defined in \eqref{th:main:tau}. It suffices to use $\lambda_k':=\frac {m+1} {2m} \lambda_k$ in Theorem \ref{th:main} and observe that $\lambda_k'\leq 1+\frac 1 m -\frac \varepsilon 2$. Similar reasoning can be repeated for the extrapolated cyclic cutter with $\sigma_k$ defined in \eqref{intro:sigma:CC}; see also Corollary \ref{th:ECCM} below. Thus we recover the framework presented in \cite{CegielskiCensor2012}.
\end{remark}

\begin{remark}\label{ex:ZNLY17}
  We now turn our attention to the results from \cite{ZhaoNgLiYao2017} which correspond to Corollary \ref{th:ESCM}. It should be mentioned here that the results from \cite{ZhaoNgLiYao2017} are presented for a more general control, which we reduce here to the $s$-intermittent one, as in the setting of this paper. Moreover, due to the nature of the algorithmic operators in \cite{ZhaoNgLiYao2017}, which are projections onto certain closed and convex supersets $C_i(x^k)\supseteq C_i$ whenever $x^k\notin C_i$, we identify these projections with cutters $U_i$; see the introduction.
  Following \cite{ZhaoNgLiYao2017}, we consider \eqref{th:ESCM:xk} with the iterative step equivalently rewritten as
  \begin{equation}\label{ex:ZNLY17:xk}
   x^{k+1} := x^k+\alpha_k\left(\sum_{i\in I_k} \omega_i U_{i} x^{k}-x^k\right),
  \quad k=0,1,2,\ldots,
  \end{equation}
  where $\alpha_k\in [\varepsilon,\ (2-\varepsilon)\tau_k(x^k)]$ for some $\varepsilon \in (0,1)$. In this case we have $\alpha_k=\lambda_k\tau_k$, $\underline{\lambda}=\varepsilon$ and $\overline{\lambda}=2-\varepsilon$. The following inequality plays an important role in the proof of \cite[Theorem 4.1]{ZhaoNgLiYao2017}:
  \begin{equation}\label{}
    \omega_n^k\alpha_k \left(2-\frac{\alpha_k}{\tau_k(x^k)}\right)=
    \omega_n^k\lambda_k\tau_k(x^k) (2-\lambda_k)
    \geq \Delta.
  \end{equation}
   It is assumed to hold for some $\Delta >0$ and a subsequence $\{x^{m_k}\}$ with $m_{k+1}-m_k\leq 2s-1$. In our case one can simply use $\Delta=\omega \underline{\lambda} (2-\overline{\lambda})=\omega\varepsilon^2$. The estimate for the linear convergence rate that follows from the proof of \cite[Theorem 4.1]{ZhaoNgLiYao2017} is
  \begin{equation}\label{}
    \|x^k - x^\infty\| \leq 2d(x^0,C)
    \left(
      \sqrt[\leftroot{-1}\uproot{2}\large \displaystyle 2s]
        {
        1-\frac{\Delta\delta^2}{2\kappa^2(1+s\Delta \delta^2 \frac{2-\varepsilon}{\varepsilon})}
        }
    \right)^{k-s+1}.
  \end{equation}

\end{remark}

\begin{remark}
  Corollary \ref{th:ECC} (i) recovers \cite[Theorem 9]{CegielskiCensor2012} and partially recovers \cite[Theorem 3.1]{NikazadMirzapour2015}, which was established for $\alpha_i$-relaxed cutters, $\alpha_i\in (0,2)$. Theorem \ref{th:main} (i) partially recovers \cite[Theorem 3.7]{NikazadMirzapour2017}, which was also established for relaxed cutters.
\end{remark}

\begin{remark}[Non-extrapolated projection methods]
In the case of the non-extrapolated string averaging projection method, with $\lambda_k=1$, our estimate \eqref{th:main:rate} for the linear rate of convergence can be written as $\|x^k-x^\infty\|\leq c q^k$, where
$q=
  \sqrt[\leftroot{-1}\uproot{2}\large \displaystyle 2s]
    {
    1-\frac{\omega}{ms\kappa^2}
    }.$
This improves \cite[Theorem 9]{BargetzReichZalas2017}, where
$
  q=
  \sqrt[\leftroot{-1}\uproot{2}\large \displaystyle 2s]
    {
    1-\frac{\omega}{2ms\kappa^2}
    }.
$
Moreover, our estimate coincides with those formulated for the cyclic and simultaneous projection methods presented in \cite[Table 1]{BargetzReichZalas2017}.
\end{remark}

\section{Numerical Simulations}

In this section we present results of numerical simulations performed on 20 systems of linear equations $Ax=b$, which we obtain by applying the Radon transform to the images presented in Figure \ref{fig:allProblems}. For the source of the test images we have used the TESTIMAGES repository\footnote{See \url{https://testimages.org} for TESTIMAGES repository.} recommended in \cite{AsuniGiachetti2013, AsuniGiachetti2014}. Every image was downgraded to a $32\times 32$ pixels size.

For every image $X \in \mathbb R^{32\times 32}$, we used the \texttt{radon}\footnote{See \url{www.mathworks.com/help/images/ref/radon} for detailed description of the MATLAB \texttt{radon} function.} function from the MATLAB Image Processing Toolbox. We chose 20 equally distributed angles starting from 0 to 180 degrees. This produces a vector $b \in \mathbb R^{835}$. We recover the matrix $A$ by applying the same \texttt{radon} function to standard basis vectors (the matrices $E_{ij}$ in this case) in the space of images $\mathbb R^{32\times 32}$. The size of the matrix is $835\times 1024$. The solution $x \in \mathbb R^{1024}$ of the linear system $Ax=b$, after a suitable rearrangement, should reproduce the image $X$.

Since each of the constraints is a hyperplane defined by $C_i:=\{x\in \mathbb R^M \mid \langle a_i, x\rangle  = b_i\}$, where $M=835$, we use metric projections $U_i:=P_{C_i}$. We recall that
\begin{equation}\label{}
  P_{C_i}x=x-\frac{\langle a_i, x\rangle-b_i}{\|a_i\|^2}a_i.
\end{equation}
We consider three methods with iterative steps of the form $x^{k+1}:=Tx^k$ and their extrapolated variants, where $x^{k+1}:=x^k+\sigma(x^k)(Tx^k-x^k)$. The extrapolation formulae are based on \eqref{intro:sigma:SC}--\eqref{intro:sigma:SA2}. The detailed description is as follows:
\begin{itemize}
  \item Simultaneous Projection Method (PM), where $\omega_i=1/M$,
  \begin{equation}\label{numerics:SimPM}
    T x:= \frac 1 M \sum_{i=1}^{M} P_{C_i}x \quad \text{and}\quad
    \sigma_1(x):=\frac{\frac 1 M \sum_{i=1}^M \|P_{C_i} x-x\|^2}{\|Tx-x\|^2}.
  \end{equation}
  \item Cyclic PM, where
  \begin{equation}\label{numerics:CycPM}
    T x:=P_{C_M}\ldots P_{C_1} x \quad \text{and}\quad
    \sigma_2(x):= \frac 1 2 +
    \frac{\sum_{i=1}^{M} \left\|P_{C_i}\ldots P_{C_1} x-P_{C_{i-1}}\ldots P_{C_1} x \right\|^2}
    {2\left\|T x-x\right\|^{2}}.
  \end{equation}
  \item String Averaging PM, with $\omega_n=1/N$,
  \begin{equation}\label{numerics:SAPM}
    T x:= \frac 1 N \sum_{n=1}^{N} Q_n, \quad \text{and} \quad Q_n x=\prod_{j\in J_n} P_{C_j}x,
  \end{equation}
  where we divide the set $\{1,2,\ldots,M\}$ into $N=84$ consecutive strings $J_n$ of length $m=10$. We visualize this process as follows:
  \begin{equation}
    \underbrace{(1,2,\ldots,10)}_{J_1},\ \underbrace{(11,12,\ldots,20)}_{J_2},\ \ldots\ ,
    \underbrace{(827, 828, \ldots,835)}_{J_{84}}.
  \end{equation}
  The length of the last five strings $J_{831},\ldots, J_{835}$ is $m-1$. Following \eqref{intro:sigma:SA1} and \eqref{intro:sigma:SA2}, we consider the extrapolation formulae
  \begin{equation}\label{numerics:sigma3}
    \sigma_3(x):=\frac{\frac 1 N \sum_{n=1}^{N}\|Q_n x-x\|^2}{\|T x-x\|^2}
  \end{equation}
  and
  \begin{equation}\label{numerics:sigma4}
  \sigma_4(x):=
  \frac
  {
    \frac 1 N \sum\limits_{n=1}^{N}
    \left(
      \left\|Q_nx-x \right\|^2
        + \sum\limits_{l=1}^{|J_n|}\|Q_{n,l}x-Q_{n,l-1}x\|^2
      \right)
  }
  {2 \|T x-x\|^2}.
\end{equation}
\end{itemize}

We apply each of the above methods to all of the 20 test problems. In each case we measure the quantity
\begin{equation}\label{numerics:measurement}
  \log_{10}\left( \frac{\max_i d(x^k,C_i)}{\max_i d(x^0,C_i)}\right)
\end{equation}
which, after averaging, we show in Figure \ref{fig:error}. In Figure \ref{fig:problem10} we present the quality of the reproduced image for problem number 10 after 500 iterations.

Using Figures \ref{fig:error} and \ref{fig:problem10}, we formulate a few observations related to all of the considered methods:
\begin{enumerate}[a)]
  \item Among all the considered methods the simultaneous PM is the slowest one whereas the cyclic PM is the fastest one. String averaging PMs lie somewhere in between.
  \item The $\sigma_1$-extrapolation applied to the simultaneous PM significantly improves the convergence properties.
  \item Both extrapolation techniques, $\sigma_3$ and $\sigma_4$, improve the convergence of the basic string averaging PM. In the considered case $\sigma_4$ has better convergence properties than $\sigma_3$.
  \item The $\sigma_2$-extrapolation does not lead to acceleration of the cyclic projection method. Nevertheless, it keeps the convergence speed.
  \item The solutions obtained for problem 10 (Figure \ref{fig:problem10}) reproduce the original image after 500 iterations although the quality for the simultaneous PM is much lower than for other methods.
\end{enumerate}

\begin{figure}[tbh]
\centering
\includegraphics[scale=0.7] {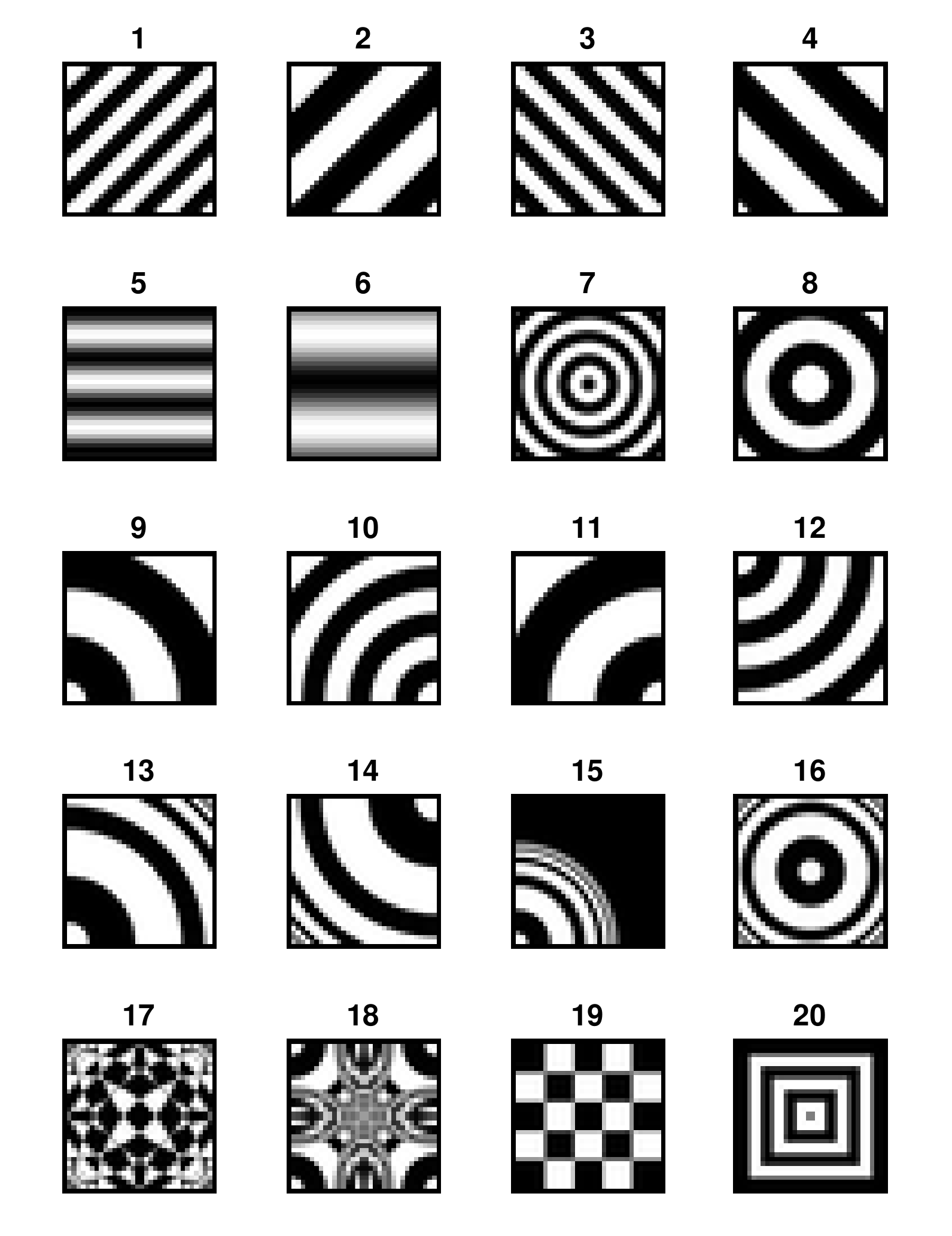}
\caption{Images that were used for numerical simulations.}
\label{fig:allProblems}
\end{figure}

\begin{figure}[h]
\centering
\includegraphics[scale=0.5] {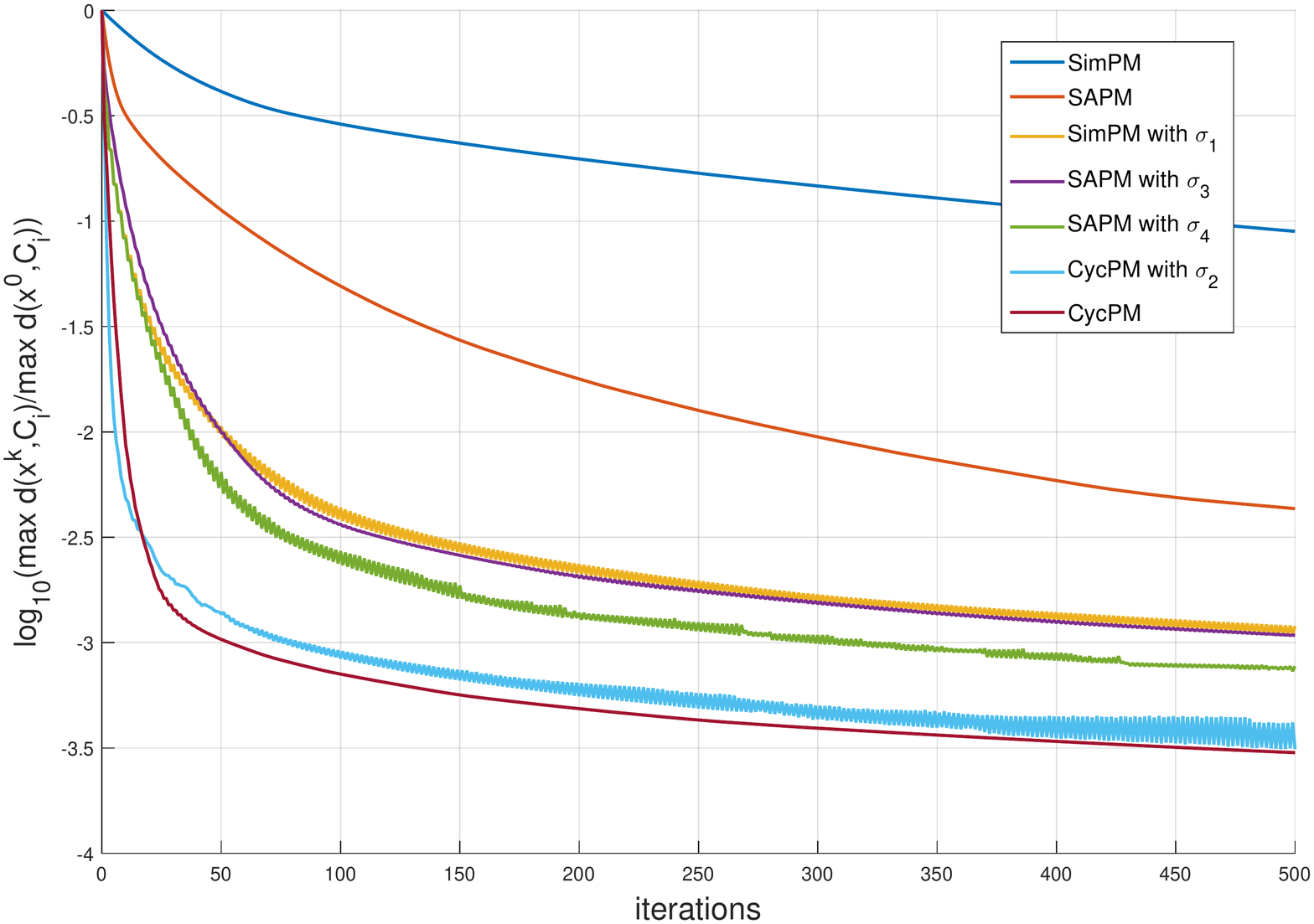}
\caption{Numerical behavior for simultaneous, cyclic and string averaging projection methods described in \eqref{numerics:SimPM}--\eqref{numerics:sigma4}. We recall that for SAPM we consider $N=84$ strings, most of them of length $m=10$. Each curve represents the average of the quantity \eqref{numerics:measurement} obtained for all 20 test problems presented in Figure \ref{fig:allProblems}.}
\label{fig:error}
\end{figure}

\begin{figure}[h]
\centering
\includegraphics[scale=0.4] {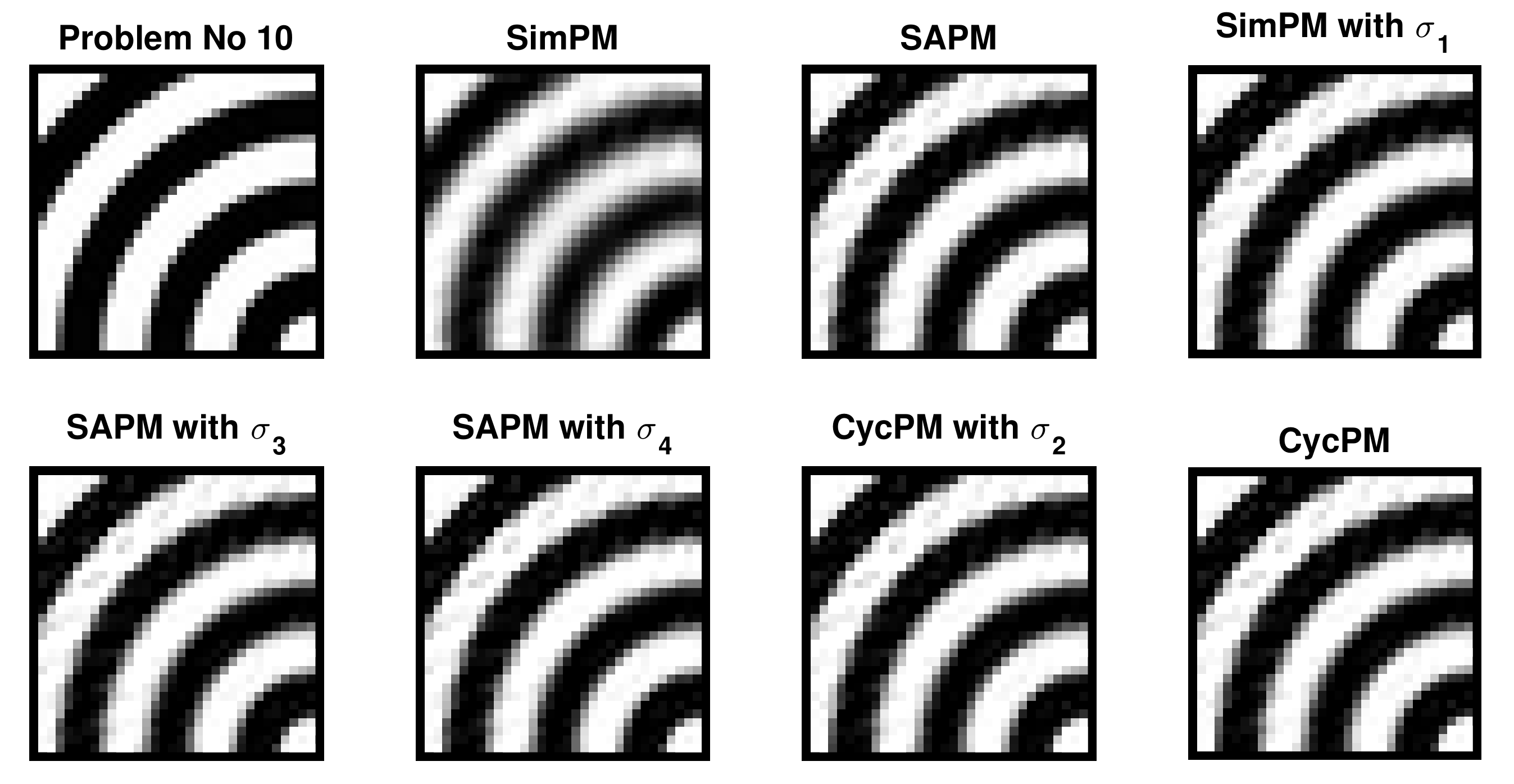}
\caption{Solutions obtained for problem number 10 after 500 iterations of simultaneous, cyclic and string averaging projection methods described in \eqref{numerics:SimPM}--\eqref{numerics:sigma4}.}
\label{fig:problem10}
\end{figure}

{\bf Acknowledgments.} This research was supported in part by the Israel Science Foundation (Grants no. 389/12 and 820/17), by the Fund for the Promotion of Research at the Technion and by the Technion General Research Fund. The first and the fourth authors were partially supported by the Austria--Israel Academic Network in Innsbruck (AIANI).

\bibliographystyle{siamplain}
\bibliography{references}

\end{document}